\documentclass[10pt,french]{smfart}

\usepackage[T1]{fontenc}
\usepackage[english,french]{babel}
\usepackage{latexsym,amscd,color}
\usepackage{amsmath,amsfonts,amssymb,mathrsfs}
\usepackage{enumerate,euscript}
\usepackage{amssymb,url,xspace,smfthm}
\input xy
\xyoption{all}

\newcommand{\id}{Id}
\renewcommand{\a}{\mathfrak a}
\DeclareMathOperator{\codim}{codim}
\DeclareMathOperator{\gal}{Gal}

\newcommand{\BibTeX}{{\scshape Bib}\kern-.08em\TeX}
\DeclareMathOperator{\Gr}{Gr}
\newcommand{\T}{\S\kern .15em\relax }
\newcommand{\AMS}{$\mathcal{A}$\kern-.1667em\lower.5ex\hbox
        {$\mathcal{M}$}\kern-.125em$\mathcal{S}$}
\newcommand{\resp}{\textit{resp}.\xspace}
\newcommand{\intersect}{\cdot\ldots\cdot}

\DeclareMathOperator{\proj}{Proj}
\DeclareMathOperator{\car}{car}

\DeclareMathOperator{\Id}{Id}

\DeclareMathOperator{\rg}{rg}

\DeclareMathOperator{\spec}{Spec}

\newcommand{\p}{\mathfrak{p}}
\DeclareMathOperator{\sym}{Sym}

\renewcommand{\O}{\mathcal{O}}
\renewcommand{\L}{\overline{L}}
\newcommand{\sm}{\mathfrak{m}}

\newcommand{\f}{\mathbb{F}}
\newcommand{\ndot}{\raisebox{.4ex}{.}}

\title{Comptage des multiplicit\'es dans une hypersurface sur un corps fini}
\alttitle{Counting multiplicities in a hypersurface over finie field}
\date{\today}
\author{Chunhui Liu}
\address{Institut de Math\'ematiques de Jussieu - Paris Rive Gauche(UMR 7586)\\
Universit\'e Paris-Diderot - Paris 7\\UFR de Math\'ematiques\\
B\^atiment Sophie Germain\\Case 7012\\75205 PARIS Cedex 13\\France}
\email{chunhui.liu@imj-prg.fr}
\urladdr{https://webusers.imj-prg.fr/~chunhui.liu/}

\begin{document}
\def\smfbyname{}
\begin{abstract}
Dans cet article, on consid\`ere un probl\`eme de comptage de multiplicit\'es. On fixe une fonction de comptage de multiplicit\'e des points rationnels dans une hypersurface d'un espace projectif sur un corps fini, et on donne une majoration de la somme de cette fonction de comptage en terme du degr\'e de l'hypersurface, de la dimension et du cardinal du corps fini. Cette majoration donne une description de la complexit\'e du lieu singulier de cette hypersurface. Afin d'obtenir la majoration, on introduit une notion appel\'ee arbre d'intersection par la th\'eorie d'intersection. On construit une suite d'intersections, telle que la multiplicit\'e d'un point rationnel singulier soit \'egale \`a celle d'une des composantes irr\'eductibles dans les intersections. Les multiplicit\'es des composantes irr\'eductibles construites ci-dessus sont major\'ees par ses multiplicit\'es dans l'arbre d'intersection.
\end{abstract}

\begin{altabstract}
In this paper, we consider a problem of counting multiplicities. We fix a counting function of multiplicity of rational points in a hypersurface of a projective space over a finite field, and we give an upper bound for the sum with respect to this counting function in terms of the degree of the hypersurface, the dimension and the cardinality of the finite field. This upper bound gives a description of the complexity of the singular locus of this hypersurface. In order to obtain this upper bound, we introduce a notion called intersection tree by intersection theory. We construct a sequence of intersections, such that the multiplicity of a singular rational point is equal to that of one of the irreducible components in these intersections. The multiplicities of these irreducible components constructed above are bounded by their multiplicities in the intersection tree.
\end{altabstract}
\maketitle

\tableofcontents

\section{Introduction}
 Dans cet article, on consid\`ere le probl\`eme de comptage des multiplicit\'es dans un sch\'ema projectif sur un corps fini. Soit $X$ un sch\'ema de type fini sur un corps fini $k$, on s'int\'eresse au probl\`eme de comptage de la forme
\[\sum_{\xi\in X(k)}f(\mu_\xi(X)),\]
o\`u $f(\ndot)$ est un polyn\^ome et $\mu_\xi(X)$ est la multiplicit\'e locale du point $\xi$ dans $X$ d\'efinie via la fonction de Hilbert-Samuel locale.

On fixe un corps fini $k=\f_q$, o\`u $q$ est une puissance d'un nombre premier $p$ (qui est la caract\'eristique du corps $k$). On consid\`ere le cas o\`u $X$ est un sous-sch\'ema ferm\'e de $\mathbb P^n_{\f_q}$. Il y a beaucoup de r\'esultats autour de la majoration du nombre des points $\f_q$-rationnels de $X$, qui signifie que l'on prend la fonction de comptage $f(\ndot)\equiv 1$ ci-dessus. Pour cela, on peut utiliser la m\'ethode analytique ou la m\'ethode de cohomologie \'etale.

Si on prend un choix non-trivial de la fonction de comptage, par exemple, on prend $f(\mu_\xi(X))$ de la forme $\mu_\xi(X)(\mu_\xi(X)-1)^t$, o\`u $t$ est un entier positif. Dans ce cas-l\`a, les m\'ethodes mentionn\'ees ci-dessus sont difficiles \`a utiliser pour cela.

\subsection{R\'esultats ant\'erieurs}
Soit $X$ une courbe plane projective r\'eduite. Dans ce cas-l\`a, le lieu singulier de $X$ est de dimension $0$ si la courbe est singuli\`ere. Soit $\delta$ le degr\'e de $X$, d'apr\`es l'exercice 5-22 dans la page 115 de \cite{Fulton2}, on a
\begin{equation}\label{intro-fulton-multiplicity}
\sum\limits_{\xi\in X}\mu_\xi(X)\left(\mu_\xi(X)-1\right)\leqslant\delta(\delta-1),
\end{equation}
qui d\'ecoule du th\'eor\`eme de B\'ezout en th\'eorie d'intersection. Plus pr\'ecis\'ement, soit $g$ le genre de la courbe plane projective $X$. Si $X$ est g\'eometriquement int\`egre, d'apr\`es le corollaire 1 dans la page 201 de \cite{Fulton2}, on a
\[g\leqslant\frac{(\delta-1)(\delta-2)}{2}-\sum_{\xi\in X}\frac{\mu_\xi(X)\left(\mu_\xi(X)-1\right)}{2}\]
par le th\'eor\`eme de Riemann-Roch sur les courbes planes.

Plus g\'en\'eralement, soit $X\hookrightarrow\mathbb P^n_k$ une hypersurface projective sur un corps alg\'ebriquement clos $k$, dont le lieu singulier est de dimension $0$. Par la m\'ethode des pinceaux de Lefschetz, une cons\'equence directe de \cite[Corollaire 4.2.1]{Laumon1975} donne
\[\sum_{\xi\in X}\mu_\xi(X)(\mu_\xi(X)-1)^{n-1}\leqslant\delta(\delta-1)^{n-1}.\]
Mais ces conditions sont trop restrictives pour un probl\`eme de comptage de multiplicit\'es.
\subsection{R\'esultat principal}
Dans cet article, on consid\`ere le probl\`eme de comptage des multiplicit\'es dans un sch\'ema sur un corps fini. On prend une fonction de comptage, et on donnera une majoration du comptage de la fonction de comptage pour une hypersurface projective. Le r\'esultat (le th\'eor\`eme \ref{main result}) est suivant:
\begin{theo}\label{main result-introduction}
Soit $X$ une hypersurface r\'eduite de degr\'e $\delta$ dans un espace projectif $\mathbb P^n_{\f_q}$, o\`u $n\geqslant2$ est un entier. Soit $s$ la dimension du lieu singulier de $X$. On a
  \begin{equation}\label{main equality}
    \sum_{\xi\in X(\f_q)}\mu_\xi(X)(\mu_\xi(X)-1)^{n-s-1}\ll_n\delta(\delta-1)^{n-s-1}\max\{\delta-1,q\}^{s}.
  \end{equation}
\end{theo}
On explicitera la constante implicite dans l'estimation \eqref{main equality} dans le th\'eor\`eme \ref{main result-introduction}.
\subsection{Motivation}
Soit $X$ un sch\'ema noeth\'erien r\'eduit qui est de dimension pure, comme le lieu r\'egulier $X^{\mathrm{reg}}$ est un ouvert dense dans $X$, on a $\codim(X,X^{\mathrm{sing}})\geqslant1$, o\`u $X^\mathrm{sing}$ est le lieu singulier de $X$.

Si on veut d\'ecrire la complexit\'e du lieu singulier de $X$ plus pr\'ecis\'ement, il n'est pas suffisant de consid\'erer seulement la dimension de $X^{\mathrm{sing}}$. Soit $X$ un sous-sch\'ema ferm\'e de $\mathbb P^n_k$. Pour d\'ecrire la complexit\'e de $X^{\mathrm{sing}}$, il faut consid\'erer la dimension de $X^{\mathrm{sing}}$, le degr\'e de $X^{\mathrm{sing}}$ et la multiplicit\'e de $X^{\mathrm{sing}}$ dans $X$ (ou les multiplicit\'es des points singuliers de $X$). Il faut choisir une fonction convenable de comptage de multiplicit\'es $f(\ndot)$ telle que $f(1)=0$.

D'apr\`es le th\'eor\`eme \ref{main result-introduction}, lorsque $X$ est une hypersurface d'un espace projectif sur un corps fini, les trois invariants ne peuvent pas \^etre trop grands simultan\'ement, qui signifie que le lieu singulier de $X$ ne peut pas \^etre "trop compliqu\'e". Dans la remarque \ref{proper counting function}, on expliquera pourquoi la fonction de comptage $\mu_\xi(X)(\mu_\xi(X)-1)^{n-s-1}$ dans l'in\'egalit\'e \eqref{main equality} est un choix convenable. Alors l'in\'egalit\'e \eqref{main equality} est une description convenable de la complexit\'e du lieu singulier de $X$ lorsque $q$ est assez grand.

\subsection{Outils principaux}
Contrairement aux m\'ethodes classiques comme par exemple la cohomologie \'etale ou la somme exponentielle, on utilise la th\'eorie d'intersection pour avoir un bon contr\^ole des multiplicit\'es.

On consid\`ere une hypersurface r\'eduite projective $X\hookrightarrow\mathbb P^n_{\f_q}$ dont le lieu singulier est de dimension $s$. Soit $Y$ un sous-sch\'ema int\`egre de $X$. Alors il existe un sous-ensemble dense $Y'$ de $Y$, tel que pour tout point $\xi\in Y'$, on ait $\mu_\xi(X)=\mu_Y(X)$. On cherche une famille $\{X_i\}_{i=1}^{n-s-1}$ d'hypersurfaces de $\mathbb P^n$ contenant $\xi$ telle que $X,X_1,\ldots,X_{n-s-1}$ s'intersectent proprement et qu'il existe une composante irr\'eductible $Y$ de l'intersection de $X,X_1,\ldots,X_{n-s-1}$ contenant $\xi$ et v\'erifiant $\mu_\xi(X)=\mu_Y(X)$. La construction de ces hypersurfaces fait intervenir des d\'eriv\'ees partielles (\'eventuellement d'ordre sup\'erieur) de l'\'equation qui d\'efinit $X$, et la construction se fait de mani\`ere r\'ecursive. Pour cela, on introduit une notion appel\'ee "arbre d'intersection" en langage de la th\'eorie des graphes, voir \S \ref{definition of intersection tree}. Un arbre d'intersection est un arbre \'etiquet\'e avec poids engendr\'es par les intersections de $X$ et certaines de ses hypersurfaces d\'eriv\'ees (voir la d\'efinition \ref{derivative hypersurface}), dont les sommets sont des sous-sch\'emas int\`egres de $X$, les \'etiquettes sont des hypersurfaces  d\'eriv\'ees, et les poids sur ses ar\^etes sont les multiplicit\'es d'intersection correspondantes \`a l'intersection du sommet et de son \'etiquette.

Comme $X$ est une hypersurface, on peut estimer la fonction $\mu_Y(X)(\mu_Y(X)-1)^{n-s-1}$ par les poids d\'efinis ci-dessus. D'apr\`es le th\'eor\`eme de B\'ezout (le th\'eor\`eme \ref{bezout}), la somme des poids peut \^etre born\'ee par le degr\'e de $X$ par rapport au fibr\'e universel de $X$.

 Pour une majoration utile du nombre de points $\f_q$-rationnels d'une composante irr\'eductible fix\'ee, on utilise l'estimation dans la proposition \ref{lineaire}.

Dans la premi\`ere section, on introduira la d\'efinition de l'arbre d'intersection afin de d\'ecrire la suite des intersections mentionn\'ee ci-dessus.  \ Dans la deuxi\`eme section, on d\'emontrera certains r\'esultats untiles de la th\'eorie d'intersection et du comptage de objets sur un corps fini. Ils sont des r\'esultats pr\'eliminaires pour le travail dans la suite. Dans la troisi\`eme section, on raisonnera par r\'ecurrence pour d\'emontrer un r\'esultat, qui est une majoration du produit des multiplicit\'es en les poids dans les arbres d'intersection. Dans la quatri\`eme section, on construira les intersection afin de d\'emontrer l'in\'egalit\'e \eqref{main equality}, et on finira la d\'emonstration.
\subsection*{Remerciment}
Ce travail fait partie de ma th\`ese pr\'epar\'ee \`a l'Universit\'e Paris Diderot - Paris 7. D'abord, je voudrais remercier profond\'ement mes directeurs de th\`ese Huayi Chen et Marc Hindry pour diriger ma th\`ese. De plus, je voudrais remercier Qing Liu pour me donner beaucoup de suggestions pour ce travail. Je voudrais remercier mes amis Yang Cao et Xiaowen Hu pour leur aide \`a ce travail.

\section{Arbre d'intersection}
Dans ce paragraphe, on introduit la notion d'arbre d'intersection dans le cadre de la th\'eorie des graphes, qui sera utilis\'ee dans l'estimation de la fonction de comptage de multiplicit\'es. Cette construction est valable dans un cadre g\'en\'eral des sch\'emas projectifs r\'eguliers sur un corps munis d'un faisceau inversible ample. Dans ce paragraphe, on fixe un corps $k$.
\subsection{D\'efinition}\label{definition of intersection tree}
Soient $Y$ un $k$-sch\'ema projectif r\'egulier et $L$ un $\O_Y$-module inversible ample. Si $X$ est un sous-sch\'ema ferm\'e de $Y$, on d\'esigne par $\deg_L(X)$ le degr\'e de $X$ par rapport au $\O_Y$-module inversible $L$, qui est d\'efini comme $\deg(c_1(L)^{\dim(X)}\cap[X])$. Soit $\delta\geqslant1$ un entier. On appelle \textit{arbre d'intersection de niveau $\delta$} sur $Y$ tout arbre $\mathscr T$ \'etiquet\'e et avec poids (sur les ar\^etes) qui v\'erifie les conditions suivantes:
  \begin{enumerate}
    \item les sommets de $\mathscr T$ sont des occurrences de sous-sch\'emas ferm\'es int\`egres de $Y$ (un sous-sch\'ema ferm\'e int\`egre de $Y$ peut appara\^itre plusieurs fois dans l'arbre);
    \item \`a chaque sommet $X$ de $\mathscr T$ est attach\'ee une \'etiquette, qui est un sous-sch\'ema ferm\'e propre de dimension pure de $Y$ ou vide;
    \item un sommet de $\mathscr T$ est une feuille si et seulement si son \'etiquette est vide;
    \item si $X$ est un sommet de $\mathscr T$ qui n'est pas une feuille, alors
    \begin{itemize}
      \item son \'etiquette $\widetilde{X}$ v\'erifie l'in\'egalit\'e $\deg_L(\widetilde{X})\leqslant \delta$ et les sous-sch\'emas ferm\'es $X$ et $\widetilde{X}$ s'intersectent proprement dans $Y$;
      \item les fils de $X$ sont pr\'ecis\'ement les composantes irr\'eductibles du produit d'intersection $X\cdot \widetilde{X}$ dans $Y$;
      \item pour tout fils $Z$ de $X$, \`a l'ar\^ete $\ell$ qui relie $X$ et $Z$ est attach\'e un poids $w(\ell)$ qui est \'egal \`a la multiplicit\'e d'intersection $i(Z;X\cdot \widetilde{X}; Y)$.
    \end{itemize}
    \end{enumerate}

  Pour un arbre d'intersection $\mathscr T$ fix\'e, on appelle \textit{sous-arbre d'intersection} de $\mathscr T$ tout sous-arbre complet de $\mathscr T$, qui est n\'ecessairement un arbre d'intersection.
   \subsubsection*{Poids d'un sommet}
Soient $Y$ un sch\'ema projectif r\'egulier sur $\spec k$, muni d'un faisceau inversible ample $L$, et $\mathscr T$ un arbre d'intersection sur $Y$. Pour tout sommet $X$ de $\mathscr T$, on d\'efinit le \textit{poids} de $X$ comme le produit des poids de tous les ar\^etes dans le chemin qui relie la racine de $\mathscr T$ et le sommet $X$, not\'e comme $w_{\mathscr T}(X)$. Si $X$ est la racine de l'arbre d'intersection, par convention $w_{\mathscr T}(X)$ est d\'efini comme $1$.
\subsubsection*{Poids d'un sous-sch\'ema ferm\'e int\`egre}
Soit $Z$ un sous-sch\'ema ferm\'e int\`egre de $Y$. On appelle \textit{poids} de $Z$ relativement \`a l'arbre $\mathscr T$ la somme des poids de toutes les occurrances de $Z$ comme sommets de $\mathscr T$, not\'e comme $W_{\mathscr T}(Z)$. Si $Z$ n'appara\^it pas dans l'arbre $\mathscr T$ comme un sommet, par convention le poids $W_{\mathscr T}(Z)$ est d\'efini comme $0$. Soit $Z$ un sommet dans l'arbre d'intersection $\mathscr T$. Lorsque l'on calcule $W_{\mathscr T}(Z)$, l'\'el\'ement $Z$ est consid\'er\'e comme un sous-sch\'ema ferm\'e int\`egre de $Y$. C'est-\`a-dire que l'on compte toutes les occurrences de $Z$ dans l'arbre d'intersection $\mathscr T$.

Dans les sous-paragraphes suivants, on rappelle quelques notions que l'on utilise dans la d\'efinition d'arbre d'intersection. Sauf mention au contraire, tous les anneaux sont suppos\'es \^etre commutatifs, unif\`eres et noeth\'eriens.

\subsection{Suite de composition}\label{length of module}
Soient $A$ un anneau et $M$ un $A$-module. On dit que $M$ est \textit{de longueur finie} s'il existe une suite d\'ecroissante de sous-modules de $M$ (appel\'ee une \textit{suite de composition} de $M$)
\[M=M_0\supsetneq M_1\supsetneq\cdots\supsetneq M_n=\{0\}\]
telle que chaque sous-quotient $M_{i-1}/M_{i}$ soit un $A$-module simple (i.e. isomorphe \`a un module quotient de $A$ par un id\'eal maximal), o\`u $i\in\{1,\ldots,n\}$. Il s'av\`ere que le nombre $n$ ne d\'epend pas du choix de la suite de composition. On l'appelle \textit{longueur} du module $M$, not\'ee comme $\ell_A(M)$, ou comme $\ell(M)$ pour simplifier. La longueur du module nul est $0$. On rappelle que, si $A$ est un anneau artinien  (i.e. un anneau noeth\'erien non-nul dont tout id\'eal premier est maximal), alors tout $A$-module de type fini est de longueur finie.  On revoie les lecteurs \`a \cite[\S2.4]{GTM150} pour plus de d\'etails.

\subsection{Multiplicit\'es de modules et d'anneaux}\label{multiplicity of local ring}
Dans cette partie, on rappelle quelques notions de multiplicit\'e dans le cadre d'alg\`ebres commutatives.
\subsubsection*{Multiplicit\'e d'un module}
Soit $A$ un anneau dont la dimension est plus grande ou \'egale \`a $1$. Soient $d$ un entier, $d\geqslant1$, $M$ un $A$-module de type fini avec $\dim_A(M)=d$, et $\a$ un id\'eal de $A$ contenu dans le radical de Jacobson de $A$ tel que l'anneau quotient $A/\a$ soit artinien. Pour tout entier naturel $m$, soit $H_{\a,M}(m)=\ell_{A/\a}(\a^mM/\a^{m+1}M)$. Il existe un polyn\^ome $P_{\a,M}$ dont le degr\'e est plus petit ou \'egal \`a $d-1$, tel que $H_{\a,M}(m)=P_{\a,M}(m)$ pour $m$ assez positif. En outre, il existe un entier $e_{\a,M}\geqslant 0$ tel que
\[P_{\a,M}(m)=e_{\a,M}\frac{m^{d-1}}{(d-1)!}+o(m^{d-1}).\]
Le nombre entier $e_{\a,M}$ est appel\'e la \textit{multiplicit\'e} de $M$ relativement \`a l'id\'eal $\mathfrak a$. Lorsque $A$ est un anneau local et $M\neq\{0\}$, on a toujours $e_{\a,M}>0$ (cf. \cite[Exercise 12.6]{GTM150}). Si $M=A$, le nombre $e_{\a,A}$ est appel\'e la \textit{multiplicit\'e} de l'id\'eal $\a$ dans $A$.

Avec les m\^emes notations ci-dessus, on consid\`ere la fonction $L_{\mathfrak a,M}(m)=\ell_{A/\mathfrak a}(M/\mathfrak a^{m+1}M)$. Il existe un polyn\^ome $Q_{\mathfrak a,M}$ dont le degr\'e est plus petit ou \'egal \`a $d$, tel que $Q_{\mathfrak a,M}(m)=L_{\mathfrak a,M}(m)$ pour $m$ assez positif. De plus, on a
\[Q_{\mathfrak a,M}(m)=e_{\a,M}\frac{m^{d}}{d!}+o(m^{d}).\]

Soient $\mathfrak a$ et $\mathfrak b$ deux id\'eaux de $A$ contenus dans le radical de Jacobson de $A$, tels que $A/\a$ et $A/\mathfrak b$ soient artiniens. Si $\mathfrak a\subseteq \mathfrak b$, d'apr\`es \cite[Chap II, \S3, a]{SamuelLocAlg}, on a $Q_{\mathfrak a,M}(m)\geqslant Q_{\mathfrak b,M}(m)$.
Alors on obtient l'in\'egalit\'e
\begin{equation}\label{smaller ideal with bigger multiplicity}
  e_{\mathfrak a,M}\geqslant e_{\mathfrak b,M}.
\end{equation}

Si $A$ est un anneau local, on peut exprimer la multiplicit\'e $e_{\mathfrak{a},M}$ comme une somme locale
  \begin{equation}
  \label{cycleofmult-alg}
    e_{\a,M}=\sum_\p\ell_{A_\p}(M_\p)\cdot e_{\a,A/\p},
  \end{equation}
  o\`u $\p$ parcourt l'ensemble des id\'eaux premiers minimaux de $A$ tels que $\dim(A)=\dim(A/\mathfrak p)$ (voir \cite[Chap. VIII, \S 7, n$^\circ$ 1, Prop. 3]{Bourbaki83} pour une d\'emonstration).

\subsubsection*{Multiplicit\'e d'un anneau local}
Soient $A$ un anneau local, $\mathfrak m$ son id\'eal maximal et $k=A/\mathfrak m$ son corps r\'esiduel. La \textit{multiplicit\'e} de $A$ est d\'efinie comme la multiplicit\'e de l'id\'eal maximal $\sm$ dans $A$. Il s'av\`ere que $e_{\mathfrak m,A}>0$ (cf. \cite[Exercise 12.6]{GTM150}).

On rappelle que l'in\'egalit\'e $\dim(A)\leqslant \dim_k(\sm/\sm^2)$ est toujours v\'erifi\'ee (cf. \cite[(12.J)]{Matsumura1}). Si on a l'\'egalit\'e $\dim(A)=\dim_k(\sm/\sm^2)$, on dit que $A$ est un \textit{anneau local r\'egulier}. Si $A$ est un anneau local r\'egulier, alors $\bigoplus\limits_{i\geqslant0}\mathfrak m^i/\mathfrak m^{i+1}$ est isomorphe \`a $\bigoplus\limits_{i\geqslant0}\sym^i_k(\mathfrak m/\mathfrak m^2)$ comme $k$-alg\`ebres gradu\'ees. Dans ce cas-l\`a, la multiplicit\'e de $A$ est $1$ (cf. \cite[\S 14]{Matsumura1}). La r\'eciproque n'est pas vraie~: il existe des anneaux locaux de multiplicit\'e $1$ qui ne sont pas r\'eguliers (voir l'exercice 2.5 dans la page 41 de \cite{Roberts98} pour un contre-exemple). Elle est vraie lorsque $\spec A$ est de dimension pure. On revoie les lecteurs \`a \cite[(40.6)]{Nagata62} pour une d\'emonstration.

\subsection{Notions de la th\'eorie d'intersection}\label{mult of a sub-scheme}
Dans cette partie, on rapplle certaines notions de la th\'eorie classique d'intersection. La r\'ef\'erence principale est \cite{Samuel}, dont l'approche est \'equivalente \`a celle de \cite{Fulton}, voir \cite[Example 7.1.1]{Fulton} et la partie e) dans la page 84 de \cite{Samuel}.
\subsubsection*{Multiplicit\'e le long d'un sous-sch\'ema ferm\'e}
Soit $X$ un sch\'ema localement noeth\'erien. Si $\xi$ est un point de $X$, on d\'esigne par $\mu_\xi(X)$ la multiplicit\'e de l'anneau local $\O_{X,\xi}$. Si $Y$ est un sous-sch\'ema ferm\'e int\`egre de $X$ dont le point g\'en\'erique est $\eta_Y$, on d\'esigne par $\O_{X,Y}$ l'anneau local $\O_{X,\eta_Y}$ pour simplifier, et on d\'esigne par $\mu_Y(X)$ la multiplicit\'e de l'anneau $\O_{X,Y}$.

\subsubsection*{Lieu r\'egulier et lieu singulier}
Soit $X$ un sch\'ema. On d\'esigne par $X^{\mathrm{reg}}$ l'ensemble des points $\xi\in X$ tels que $\O_{X,\xi}$ soit un anneau local r\'egulier, appel\'e le \textit{lieu r\'egulier} de $X$. Si $X^{\mathrm{reg}}=X$, on dit que $X$ est un \textit{sch\'ema r\'egulier}. Soit en outre $X^{\mathrm{sing}}$ le compl\'ementaire $X\smallsetminus X^{\mathrm{reg}}$, appel\'e le \textit{lieu singulier} de $X$. Si $X$ est localement de type fini sur le spectre d'un corps, l'ensemble $X^{\mathrm{reg}}$ est un ouvert Zariski de $X$ (cf. \cite[Corollary 8.16, Chap. II]{GTM52}), et donc l'ensemble des points de multiplicit\'e $1$ est dense dans $X$ si $X$ est irr\'eductible et $X^{\mathrm{reg}}\neq\emptyset$.
\subsubsection*{Multiplicit\'e d'intersection}
Soit $X$ un sch\'ema noeth\'erien de dimension finie. On dit que $X$ est de \textit{dimension pure} si toutes les composantes irr\'eductibles de $X$ ont la m\^eme dimension.

Soit $k$ un corps. Soit $Y$ un sch\'ema r\'egulier de type fini sur $\spec k$ tel que le morphisme canonique $Y\rightarrow\spec k$ soit s\'epar\'e, et soient $X_1,\ldots,X_r$ des sous-sch\'emas ferm\'es de dimension pure de $Y$. On d\'esigne par $\Delta:Y\rightarrow Y^{\times_kr}$ le morphisme diagonal. Il s'av\`ere que le produit fibr\'e de $\Delta(Y)$ et $X_1\times_k\cdots\times_kX_r$ sur $Y^{\times_kr}$ est isomorphe \`a l'intersection sch\'ematique $\bigcap_{i=1}^rX_i$. Ainsi on peut consid\'erer $\bigcap_{i=1}^rX_i$ comme un sous-sch\'ema ferm\'e de $X_1\times_k\cdots\times_kX_r$. Soit $\mathcal I$ le faisceau d'id\'eaux de $\mathcal O_{X_1\times_k\cdots\times_kX_r}$ correspondant \`a $\bigcap_{i=1}^rX_i$.

Soit $M$ une composante irr\'eductible de $\bigcap_{i=1}^rX_i$ consid\'er\'e comme un sous-sch\'ema ferm\'e int\`egre de $Y$. On d\'esigne par $\Delta(M)$ le sous-sch\'ema ferm\'e int\`egre de $X_1\times_k\cdots\times_kX_r$ l'image de $M$ par le morphisme diagonal (qui est une immersion ferm\'ee car $Y$ est s\'epar\'e sur $\spec k$). Soit $\eta_M$ le point g\'en\'erique de $\Delta(M)$. L'id\'eal $\mathcal I_{\eta_M}$ est appel\'e \textit{l'id\'eal diagonal} de l'anneau $\mathcal O_{X_1\times_k\cdots\times_kX_r,\Delta(M)}$. On d\'efinit la \textit{multiplicit\'e d'intersection} de $X_1,\ldots,X_r$ en $M$ comme la multiplicit\'e de l'id\'eal $\mathcal I_{\eta_M}$ dans l'anneau local $\mathcal O_{X_1\times_k\cdots\times_kX_r,\Delta(M)}$, not\'ee comme \[i(M;X_1\intersect X_r;Y).\]
Si un sch\'ema int\`egre $N$ de $Y$ n'est pas une composante irr\'eductible de $X_1\cap\cdots\cap X_r$, on d\'efinit
\[i(N;X_1\intersect X_r;Y)=0\]
par convention. On revoie les lecteurs \`a la page 148 de \cite{FGA-Weil} et la page 77 de \cite{Samuel} pour plus de d\'etails de cette d\'efinition (voir aussi les chapitres 7 et 8 de \cite{Fulton} pour une autre d\'efinition \'equivalente).

\subsubsection*{Composantes propres}
Soit $k$ un corps. Soient $Y$ un $k$-sch\'ema r\'egulier s\'epar\'e de type fini, et $X_1,\ldots,X_r$ des sous-sch\'emas ferm\'es de dimension pure de $Y$. On d\'esigne par $\mathcal C(X_1\intersect X_r)$ l'ensemble des composantes irr\'eductibles de l'intersection sch\'ematique $X_1\cap\cdots\cap X_r$. En particulier, si $X$ est un sous-sch\'ema ferm\'e de dimension pure de $Y$, alors $\mathcal C(X)$ d\'esigne l'ensemble des composantes irr\'eductibles de $X$. Sauf sp\'ecifiquement mentionn\'e, toute composante irr\'eductible dans $\mathcal C(X_1\intersect X_r)$ ou $\mathcal C(X)$ est consid\'er\'ee comme un sous-sch\'ema ferm\'e int\`egre de $Y$.

On rappelle que l'on a (cf. \cite[Chap. III, Prop. 17]{SerreLocAlg})
\[\dim(M)\geqslant\dim(X_1)+\cdots+\dim(X_r)-(r-1)\dim(Y)\]
pour tout $M\in\mathcal C(X_1\intersect X_r)$. On dit que les sch\'emas $X_1,\ldots,X_r$ \textit{s'intersectent proprement} en $M$ dans $Y$, ou encore $M$ est une \textit{composante propre} de l'intersection $X_1\cdot\ldots\cdot X_r$ dans $Y$, si $M\in\mathcal C(X_1\intersect X_r)$ et si l'\'egalit\'e
\[\dim(M)=\dim(X_1)+\cdots+\dim(X_r)-(r-1)\dim(Y)\]
est v\'erifi\'ee. On dit que $X_1,\ldots,X_r$ \textit{s'intersectent proprement} si tout \'el\'ement $M\in\mathcal C(X_1\intersect X_r)$ est une composante propre de l'intersection $X_1\cdot\ldots\cdot X_r$ dans $Y$.

\section{Estimation de poids des arbres d'intersection}
\subsection{\'Enonc\'e du th\'eor\`eme}
Dans tout le paragraphe, on fixe un corps $k$, un entier $n \geqslant 1$ et un espace vectoriel $E$ de rang $n+1$ sur $k$. On d\'efinit l'espace projectif $\mathbb P(E)$ comme le sch\'ema qui repr\'esente le foncteur de la cat\'egorie des $k$-alg\`ebres commutatives dans la cat\'egorie des ensembles, qui envoie toute $k$-alg\`ebre commutative $A$ sur l'ensemble des $A$-modules quotients de $E\otimes_{k}A$ qui sont projectifs de rang $1$. De plus, on d\'esigne par $\mathbb P^n_k$ l'espace projectif $\mathbb P(k^{n+1})$ pour simplifier, ou par $\mathbb P^n$ s'il n'y a pas d'ambigu\"it\'e sur $k$. Si $L$ est le faisceau inversible universel $\O_{\mathbb P(E)}(1)$, le degr\'e de $X$ par rapport \`a $\O_{\mathbb P(E)}(1)$ est not\'e comme $\deg(X)$ pour simplifier.

Soient $\{X_i\}_{i=1}^r$ une famille de sous-sch\'emas ferm\'es de dimension pure de $\mathbb P(E)$ qui s'intersectent proprement dans $\mathbb P(E)$ (voir \S\ref{mult of a sub-scheme} pour la d\'efinition). On \'etablira le th\'eor\`eme suivant, qui peut \^etre consid\'er\'e comme une majoration du produit des multiplicit\'es locales de $X_1,\ldots,X_r$ en fonction des arbres d'intersections.
\begin{theo}\label{chongshu2}
 On suppose que $k$ est un corps parfait. Soient $\{X_i\}_{i=1}^r$ une famille de sous-sch\'emas ferm\'es de dimension pure de $\mathbb P(E)$ qui s'intersectent proprement dans $\mathbb P(E)$. Pour tout composante irr\'educitble $Y\in\mathcal C(X_1\intersect X_r)$. Soit un arbre d'intersection $\mathscr T_Y$ ayant $Y$ comme racine. On consid\`ere un sommet $M$ dans les arbres d'intersection $\{\mathscr T_Y\}_{Y\in\mathcal C(X_1\intersect X_r)}$ v\'erifiant: pour tout sommet $Z$ dans $\{\mathscr T_Y\}_{Y\in\mathcal C(X_1\intersect X_r)}$, si $M$ est un sous-sch\'ema propre de $Z$, alors il existe un descendant de $Z$ qui est une occurrence de $M$ comme sch\'emas. Alors l'in\'egalit\'e suivante est satisfaite :
  \begin{equation}\label{no auxillary scheme}
    \sum_{Y\in\mathcal C(X_1\intersect X_r)}W_{\mathscr T_Y}(M)i(Y;X_1\intersect X_r;\mathbb P(E))\geqslant \mu_{M}(X_1)\cdots\mu_{M}(X_r),
  \end{equation}
  o\`u l'expression $\mu_M(X_i)$ d\'esigne la multiplicit\'e de l'anneau local de $X_i$ en le point g\'en\'erique de $M$.
\end{theo}

On rappelle que la \textit{profondeur} d'un sommet est d\'efinie comme la longueur du chemin qui relie ce sommet et la racine de l'arbre. En outre, la \textit{profondeur} d'un arbre est d\'efinie comme la valeur maximale des profondeurs de ses sommets.
\begin{exem}
 On va donner un exemple de l'op\'eration dans le th\'eor\`eme \ref{chongshu2}. On prend $\mathbb P(E)=\mathbb P^4_k=\proj\left(k[T_0,T_1,T_2,T_3,T_4]\right)$ comme le sch\'ema de base. Soient
  \[X_1=\proj\left(k[T_0,T_1,T_2,T_3,T_4]/(T_4)\right),\]
   et
   \[X_2=\proj\left(k[T_0,T_1,T_2,T_3,T_4]/(T_3(T_0^2T_1-T_2^3+T_2^2T_1))\right).\]
    Alors on a $\deg(X_1)=1$ et $\deg(X_2)=4$. Les sch\'emas $X_1$ et $X_2$ s'intersectent proprement dans $\mathbb P^4_k$. L'intersection de $X_1$ et $X_2$ admet deux composantes irr\'eductibles, not\'ees comme $Y_1$ et $Y_2$. Soient
    \[Y_1=\proj\left(k[T_0,T_1,T_2,T_3,T_4]/(T_0^2T_1-T_2^3+T_2^2T_1,T_4)\right)\]
    un \'el\'ement dans $\mathcal C(X_1\cdot X_2)$, et
     \[Y_2=\proj\left(k[T_0,T_1,T_2,T_3,T_4]/(T_3,T_4)\right)\]
     un autre \'el\'ement dans $\mathcal C(X_1\cdot X_2)$. Alors par d\'efinition, on a
     \[i(Y_1;X_1\cdot X_2;\mathbb P^4_k)=1,\;\deg(Y_1)=3;\]
     et
     \[i(Y_2;X_1\cdot X_2;\mathbb P^4_k)=1,\;\deg(Y_2)=1.\]

    On va construire deux arbres d'intersection suivants dont les racines sont $Y_1$ et $Y_2$.
  
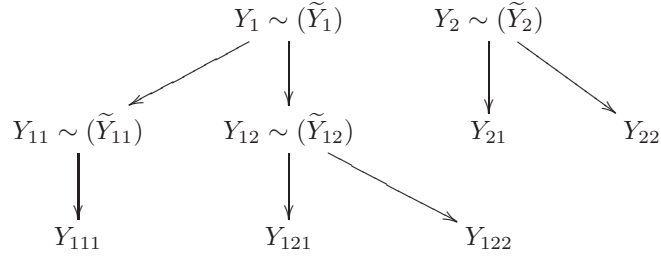
\[\xymatrix{ &Y_1\sim(\widetilde{Y}_1)\ar[ld]\ar[d]&Y_2\sim(\widetilde{Y}_2)\ar[d]\ar[rd]& \\ Y_{11}\sim(\widetilde{Y}_{11})\ar[d] &Y_{12}\sim(\widetilde{Y}_{12})\ar[rd]\ar[d]&Y_{21}&Y_{22}\\ Y_{111}& Y_{121} & Y_{122} & }\]

  On suppose que l'\'etiquette de $Y_1$ est l'hypersurface
  \[\widetilde{Y}_1=\proj\left(k[T_0,T_1,T_2,T_3,T_4]/(T_1T_3)\right),\;\deg(\widetilde{Y}_1)=2,\]
   et l'\'etiquette de $Y_2$ est
   \[\widetilde{Y}_2=\proj\left(k[T_0,T_1,T_2,T_3,T_4]/(T_2,T_0(T_1+T_0))\right).\;\deg(\widetilde{Y}_2)=2.\]
   Alors on peut confirmer que l'intersection de $Y_1$ et $\widetilde{Y}_1$ et l'intersection $Y_2$ et $\widetilde{Y}_2$ sont propres.

   Dans la suite, on consid\`ere l'arbre d'intersection ayant $Y_2$ comme racine. En fait, il a deux composantes irr\'eductbles, not\'ees comme $Y_{21}$ et $Y_{22}$. Par d\'efinition, on obtient
   \[Y_{21}=[0:1:0:0:0],\;i(Y_{21};Y_2\cdot\widetilde{Y}_2;\mathbb P^4_k)=1;\]
   et
   \[Y_{22}=[1:-1:0:0:0],\;i(Y_{22};Y_2\cdot\widetilde{Y}_2;\mathbb P^4_k)=1.\]

  Pour l'arbre dont la racine est $Y_1$, l'ensemble $\mathcal C(Y_1\cdot\widetilde{Y}_1)$ a deux \'el\'ements, not\'es comme $Y_{11}$ et $Y_{12}$ respectivement. On suppose
    \[Y_{11}=\proj\left(k[T_0,T_1,T_2,T_3,T_4]/(T_1,T_2,T_4)\right)\]
    et
    \[Y_{12}=\proj\left(k[T_0,T_1,T_2,T_3,T_4]/(T_0^2T_1-T_2^3+T_2^2T_1,T_3,T_4)\right)\]
    Alors on a
    \[i(Y_{11};Y_1\cdot\widetilde{Y}_1;\mathbb P^4_k)=3,\;\deg(Y_{11})=1;\]
    et
    \[i(Y_{12};Y_1\cdot\widetilde{Y}_1;\mathbb P^4_k)=1,\;\deg(Y_{12})=3.\]
    L'\'egalit\'e $i(Y_{11};Y_1\cdot\widetilde{Y}_1;\mathbb P^4_k)=3$ est d'apr\`es que l'anneau local en $Y_{11}$ est Cohen-Macaulay, par \cite[Proposition 7.1]{Fulton}, cette multiplicit\'e d'intersection est \'egale \`a $\ell(\O_{Y_1\cap\widetilde{Y}_1,Y_{11}})$, qui est \'egal \`a $3$. Soient
    \[\widetilde{Y}_{11}=\proj\left(k[T_0,T_1,T_2,T_3,T_4]/(T_0+T_3)\right),\;\deg(\widetilde{Y}_{11})=1\]
    l'\'etiquette de $Y_{11}$, et
    \[\widetilde{Y}_{12}=\proj\left(k[T_0,T_1,T_2,T_3,T_4]/(T_2)\right),\;\deg(\widetilde{Y}_{12})=1\]
    l'\'etiquette de $Y_{12}$.
    Alors on obtient que l'intersection de $Y_{11}$ et $\widetilde{T}_{11}$ admet une composante irr\'eductible, et que l'intersection de $Y_{12}$ et $\widetilde{Y}_{12}$ admet deux composantes irr\'eductibles not\'ees comme $Y_{121}$ et $Y_{122}$. De plus, on a
    \[Y_{111}=[1:0:0:-1:0],\;i(Y_{111};Y_{11}\cdot\widetilde{Y}_{11};\mathbb P^4_k)=1,\]
    et
    \[Y_{121}=[0:1:0:0:0],\;i(Y_{121};Y_{12}\cdot\widetilde{Y}_{12};\mathbb P^4_k)=2,\]
    et
    \[Y_{122}=[1:0:0:0:0],\;i(Y_{122};Y_{12}\cdot\widetilde{Y}_{12};\mathbb P^4_k)=1\]
    par d\'efinition directement.

  Soit $M=[0:1:0:0:0]$. On peut confirmer que les sommets $Y_{121}=Y_{21}=M$ satisfont les conditions dans le th\'eor\`eme \ref{chongshu2} consid\'er\'es comme deux sous-sch\'emas int\`egres de $\mathbb P^4_k$. Dans cet exemple, Le c\^ot\'e gauche de l'in\'egalit\'e \eqref{no auxillary scheme} \'egal \`a
    \begin{eqnarray*}
      & &i(Y_1;X_1\cdot X_2;\mathbb P^4_k)i(Y_{12};Y_1\cdot\widetilde{Y}_1;\mathbb P^4_k)i(Y_{121};Y_{12}\cdot\widetilde{Y}_{12};\mathbb P^4_k)\\
      & &+i(Y_2;X_1\cdot X_2;\mathbb P^4_k)i(Y_{21};Y_2\cdot\widetilde{Y}_2;\mathbb P^4_k)\\
      &=&3.
    \end{eqnarray*}
    De plus, comme l'hypersurface $X_1$ est r\'eguli\`ere, on a
    \[\mu_M(X_1)=1;\]
    en consid\'erant le d\'eveloppement de Taylor de l'\'equation qui d\'efinit l'hypersurface $X_2$, on obtient
    \[\mu_{M}(X_2)=3.\]
    Alors le c\^ot\'e droite de l'in\'egalit\'e \eqref{no auxillary scheme} \'egal \`a
    \[\mu_M(X_1)\mu_M(X_2)=3.\]
    Donc on a l'in\'egalit\'e
    \begin{eqnarray*}
      & &i(Y_1;X_1\cdot X_2;\mathbb P^4_k)i(Y_{12};Y_1\cdot\widetilde{Y}_1;\mathbb P^4_k)i(Y_{121};Y_{12}\cdot\widetilde{Y}_{12};\mathbb P^4_k)\\
      & &+i(Y_2;X_1\cdot X_2;\mathbb P^4_k)i(Y_{21};Y_2\cdot\widetilde{Y}_2;\mathbb P^4_k)\\
      &\geqslant&\mu_M(X_1)\mu_M(X_2),
    \end{eqnarray*}
    ce qui est un exemple du th\'eor\`eme \ref{chongshu2}.
\end{exem}
\subsection{R\'esultats pr\'eliminaires}
Dans cette partie, on introduira certains r\'esultats pr\'eliminaires pour la d\'emonstration du th\'eor\`eme \ref{chongshu2}.

\subsubsection*{Commutativit\'e et associativit\'e d'intersection}
La multiplicit\'e d'intersection satisfait \`a la commutativit\'e et la associativit\'e au sens suivant. On revoie les lecteurs \`a \cite[Proposition 8.1.1]{Fulton} pour une d\'emonstration.
\begin{theo}\label{comm-ass}
Soient $X_1,X_2,X_3$ trois sous-sch\'emas ferm\'es de dimension pure d'un sch\'ema s\'epar\'e r\'egulier $Y$ de type fini sur $\spec k$. On a les propri\'et\'es suivantes:
  \begin{description}
    \item[(i). (commutativit\'e)] pour tout $M\in\mathcal C(X_1\cdot X_2)=\mathcal C(X_2\cdot X_1)$, on a
    \[i(M;X_1\cdot X_2;Y)=i(M;X_2\cdot X_1;Y);\]
    \item[(ii). (associativit\'e)] si $X_1,X_2,X_3$ s'intersectent proprement en $M\in\mathcal C(X_1\cdot X_2\cdot X_3)$, alors on a:
    \begin{eqnarray*}
  i(M;X_1\cdot X_2\cdot X_3;Y)&=&\sum_{P\in \mathcal C(X_1\cdot X_2)}i(M;P\cdot X_3;Y)\cdot i(P;X_1\cdot X_2;Y)\\
  &=&\sum_{Q\in\mathcal C(X_2\cdot X_3)}i(M;Q\cdot X_1;Y)\cdot i(Q;X_2\cdot X_3;Y),
\end{eqnarray*}
voir \S\ref{mult of a sub-scheme} pour les notations de $\mathcal C(X_1\cdot X_2\cdot X_3)$, $\mathcal C(X_1\cdot X_2)$ et $\mathcal C(X_2\cdot X_3)$.
\end{description}
  \end{theo}
\subsubsection*{Th\'eor\`eme de B\'ezout}
Le th\'eor\`eme de B\'ezout est une description de la complexit\'e d'une intersection propre dans $\mathbb P(E)$ en termes de degr\'es rapport au fibr\'e universel.
\begin{theo}[le th\'eor\`eme de B\'ezout]\label{bezout}
  Soient $X_1,\ldots,X_r$ des sous-sch\'emas ferm\'es de dimension pure de $\mathbb P(E)$, qui s'intersectent proprement. Alors on a
  \begin{equation*}
    \sum_{Z\in\mathcal C(X_1\intersect X_r)}i(Z;X_1\intersect X_r;\mathbb P(E))\deg(Z)=\deg(X_1)\cdots\deg(X_r).
  \end{equation*}
  \end{theo}
On revoie les lecteurs \`a \cite[Proposition 8.4]{Fulton} pour plus d\'etails, voir l'\'egalit\'e (1) dans la page 145 de \cite{Fulton} aussi.
\subsubsection*{Invariance par extension de corps}
Soient $X$ un sch\'ema sur le corps $k$ et $k'/k$ une extension de corps. On d\'esigne par $X_{k'}$ le produit fibr\'e $X\times_{\spec k}\spec k'$. De plus, soit $E$ un espace $k$-vectoriel. On d\'esigne par $E_{k'}$ l'espace $k'$-vectoriel $E\otimes_{k}k'$.

Soient $X_1,\ldots,X_r$ des sous-sch\'emas ferm\'es de $\mathbb P(E)$, $M\in\mathcal C(X_1\intersect X_r)$, et $M'\in\mathcal C(M_{k'})$ (voir \S\ref{mult of a sub-scheme} pour les notations). On d\'emontrera que $M'\in\mathcal C(X_{1,k'}\intersect X_{r,k'})$ dans le lemme \ref{irreducible component change after extension of fields}. De plus, lorsque $k'/k$ est une extension galoisienne finie, on \'etudiera une relation entre $i(M;X_1\intersect X_r;\mathbb P(E))$ et $i(M';X_{1,k'}\intersect X_{r,k'};\mathbb P(E_{k'}))$. Soient $X$ un sous-sch\'ema ferm\'e de $\mathbb P(E)$, $M$ un sous-sch\'ema ferm\'e int\`egre de $X$, et $M'\in\mathcal C(M_{k'})$. On obtient une relation entre $\mu_M(X)$ et $\mu_{M'}(X_{k'})$ si $k'/k$ est galoisienne finie.

\begin{prop}\label{sumofmult}
   Soient $X$ un sous-sch\'ema ferm\'e de $\mathbb P(E)$ de dimension pure, et $Z$ un sous-sch\'ema ferm\'e int\`egre de $X$. Alors on a
    \begin{equation*}
    \deg(X)=\sum_{X'\in\mathcal C(X)}\ell_{\O_{X,X'}}(\O_{X,X'})\deg(X').
  \end{equation*}
  et
  \begin{equation*}
    \mu_Z(X)=\sum_{X'\in\mathcal C(X)}\ell_{\O_{X,X'}}(\O_{X,X'})\mu_Z(X').
  \end{equation*}
\end{prop}
\begin{proof}
Si on d\'efinit le degr\'e d'un sch\'ema projectif par la multiplicit\'e d'un id\'eal (cf. \cite[Chap. I, Proposition 7.5]{GTM52}), les deux \'egalit\'e sont des cons\'equences directes de l'\'egalit\'e \eqref{cycleofmult-alg}. Si on prend la d\'efinition de degr\'e d'un sch\'ema projectif de dimension pure par le nombre d'intersection comme ci-dessus, on revoie les lecteurs \`a \cite[Example 2.5.2 (b)]{Fulton} pour une d\'emonstration.
\end{proof}

La proposition \ref{sumofmult} sera utilis\'ee dans les d\'emonstrations des r\'esultats au-dessous.
\begin{lemm}\label{irreducible component change after extension of fields}
  Soit $k$ un corps.  Soient $X_1,\ldots,X_r$ des sous-sch\'emas ferm\'es de $\mathbb P(E)$, et $Y\in\mathcal C(X_1\intersect X_r)$. Soit $k'/k$ une extension de corps. Alors pour toute composante irr\'eductible $Y'\in\mathcal C(Y_{k'})$, on a $Y'\in\mathcal C(X_{1,k'}\intersect X_{r,k'})$. De plus, l'application canonique
  \[\bigsqcup_{Y\in\mathcal C(X_1\intersect X_r)}\mathcal C(Y_{k'})\rightarrow\mathcal C(X_{1,k'}\intersect X_{r,k'})\]
  est une bijection. Autrement dit, pour tout $Y'\in\mathcal C(X_{1,k'}\intersect X_{r,k'})$, il existe un et un unique $Y\in\mathcal C(X_1\intersect X_r)$ tel que $Y'$ soit une composante irr\'eductible de $Y_{k'}$.
\end{lemm}
\begin{proof}
   D'apr\`es \cite[Proposition 3.2.7]{LiuQing}, pour tout $Y'\in\mathcal C(Y_{k'})$, on a $\dim(Y')=\dim(Y_{k'})=\dim (Y)$.

   Soit $Z'\in\mathcal C(X_{1,k'}\intersect X_{r,k'})$. On consid\`ere le morphisme de projection $\pi':\mathbb P(E_{k'})\rightarrow \mathbb P(E)$. Par d\'efinition, on a $\pi'(Z')\subseteq \bigcap_{i=1}^rX_i$, alors on en d\'eduit que le sch\'ema $\pi'(Z')$ est contenu dans un \'el\'ement dans $\mathcal C(X_1\intersect X_r)$. Par le fait que $Z'\subseteq\pi'(Z')_{k'}$, on obtient que $Z'$ est contenu dans un $Y'\in\mathcal C(Y_{k'})$, o\`u $Y\in\mathcal C(X_1\intersect X_r)$.

   Le morphisme $\spec k'\rightarrow \spec k$ \'etant fini et fid\`element plat, il en est de m\^eme du morphisme de projection $\pi:\mathbb P(E_{k'})^{\times_{k'}r}\rightarrow \mathbb P(E)^{\times_kr}$ (cf. \cite[Corollaire 2.2.13 (i)]{EGAIV_2}). Soit $Y'\in\mathcal C(Y_{k'})$. Soient $\eta$ et $\eta_0$ les points g\'en\'eriques de $\Delta(Y)$ et $\Delta(Y')$ respectivement, o\`u les $\Delta$ d\'esignent les morphismes diagonaux. D'apr\`es \cite[Proposition 2.3.4 (i)]{EGAIV_2}, le morphisme de projection $\pi$ envoie $\eta_0$ sur $\eta$. Si $Z'\in\mathcal C(X_{1,k'}\intersect X_{r,k'})$ qui est contenu dans $\Delta(Y')$, alors on a $\pi(\Delta(Z'))=\Delta(Y)$. Encore par \cite[Proposition 2.3.4 (i)]{EGAIV_2}, on obtient que la codimension de $Z'$ dans $\mathbb P(E_{k'})$ est born\'ee sup\'erieurement par celle de $Y$ dans $\mathbb P(E)$, d'o\`u $\dim(Z')\geqslant\dim(Y) = \dim(Y_{k'})$ puisque les sch\'emas alg\'ebriques sont cat\'enaires. Donc on obtient $Z'= Y'$.
\end{proof}
La proposition suivante est l'invariance de la multiplicit\'e d'intersection par une extension de corps finie. Certaines id\'ees de la d\'emonstration proviennent de \cite{Nowak_1998}.
\begin{prop}\label{basechangemultinter}
   Soient $X_1,\ldots,X_r$ des sous-sch\'emas ferm\'es de $\mathbb P(E)$, et $Y\in\mathcal C(X_1\intersect X_r)$. Soit $k'/k$ une extension galoisienne finie de corps. Alors pour toute composante irr\'eductible $Y'\in\mathcal C(Y_{k'})$ (on a $Y'\in\mathcal C(X_{1,k'}\intersect X_{r,k'})$ d'apr\`es le lemme \ref{irreducible component change after extension of fields}), l'\'egalit\'e
  \[i(Y';X_{1,k'}\intersect X_{r,k'};\mathbb P(E_{k'}))=i(Y;X_1\intersect X_r;\mathbb P(E))\]
  est v\'erifi\'ee.
\end{prop}
\begin{proof}
    D'abord, on consid\`ere le diagramme suivant:
    \[\xymatrix{\relax\mathbb P(E_{k'})\ar@/^7mm/[drr]^{\Delta_{\mathbb P(E_{k'})/k}}\ar@/_7mm/[ddr]_\pi\ar[dr]^{\Delta_{\mathbb P(E_{k'})/\mathbb P(E)}}& & \\&\mathbb P(E_{k'})^{\times_{\mathbb P(E)}r}\ar[r]\ar[d]\ar@{}|-{\square}[dr]&\mathbb P(E_{k'})^{\times_{k}r}\ar[d]\\
    &\mathbb P(E)\ar[r]^{\Delta_{\mathbb P(E)/k}}&\mathbb P(E)^{\times_{k}r},}\]
o\`u $\Delta_{\mathbb P(E_{k'})/k}$, $\Delta_{\mathbb P(E_{k'})/\mathbb P(E)}$, et $\Delta_{\mathbb P(E)/k}$ sont des morphismes diagonaux, et $\pi$ est le morphisme canonique obtenu par le changement de base $\spec k'\rightarrow \spec k$.

    D'apr\`es \cite[Proposition (1.4.5), Chap. 0]{NouveauEGA1} et \cite[Proposition (1.4.8), Chap. 0]{NouveauEGA1}, le diagramme plus haut est commutatif.

   Comme l'extension $k'/k$ est s\'eparable, le morphisme canonique $\pi:\:\mathbb P(E_{k'})\rightarrow \mathbb P(E)$ est \'etale et fini. De plus, le morphisme $\Delta_{\mathbb P(E_{k'})/\mathbb P(E)}$ est une section du morphisme de projection (\`a une coordonn\'ee arbitraire)
   \[\mathbb P(E_{k'})^{\times_{\mathbb P(E)}r}\rightarrow\mathbb P(E_{k'}),\]
   o\`u la projection ci-dessus est \'etale et s\'epar\'ee.
   D'apr\`es \cite[Corollary 3.12]{Milne}, pour tout sous-sch\'ema ferm\'e de $\mathbb P(E_{k'})$, le morphisme $\Delta_{\mathbb P(E_{k'})/\mathbb P(E)}$ est un isomorphisme dans toute composante connnexe de ce sous-sch\'ema ferm\'e. D'o\`u l'on obtient que pour tout sous-sch\'ema ferm\'e int\`egre $M$ de $\mathbb P(E)$, et tout $M'\in\mathcal C(M_{k'})$, l'id\'eal diagonal de l'anneau $\O_{X_{1,k'}\times_k\cdots\times_{k}X_{r,k'},\Delta_{\mathbb P(E_{k'})/k}(M')}$ est un module obtenu de l'id\'eal diagonal de l'anneau $\O_{X_{1}\times_k\cdots\times_{k}X_{r},\Delta_{\mathbb P(E)/k}(M)}$ par extension des scalaires.

   De plus, d'apr\`es \cite[Proposition(1.4.8), Chap. 0]{NouveauEGA1}, le diagramme
    \[\xymatrix{\relax \mathbb P(E_{k'})\ar[r]^{\id}\ar[d]^{\Delta_{\mathbb P(E_{k'})/k'}} & \mathbb P(E_{k'})\ar[d]^{\Delta_{\mathbb P(E_{k'})/k}}\\
    \mathbb P(E_{k'})^{\times_{k'}r}\ar[r]\ar[d]\ar@{}|-{\square}[dr] &\mathbb P(E_{k'})^{\times_{k}r}\ar[d]\\
    \spec k'\ar[r]&\spec (k'^{\otimes_kr}),}\]
est commutatif, o\`u $\Delta_{\mathbb P(E_{k'})/k'}$ et $\Delta_{\mathbb P(E_{k'})/k}$ sont des morphismes diagonaux. D'o\`u l'on obtient que pour tout sous-sch\'ema ferm\'e int\`egre $M'\in\mathcal C(M_{k'})$ d\'efini ci-dessus, l'id\'eal diagonal de l'anneau $\O_{X_{1,k'}\times_{k'}\cdots\times_{k'}X_{r,k'},\Delta_{\mathbb P(E_{k'})/k'}(M')}$ est un module obtenu de l'id\'eal diagonal de l'anneau $\O_{X_{1,k'}\times_k\cdots\times_{k}X_{r,k'},\Delta_{\mathbb P(E_{k'})/k}(M')}$ par extension des scalaires sous le changement de base ci-dessus. Par cons\'equent, l'id\'eal diagonal de l'anneau $\O_{X_{1,k'}\times_{k'}\cdots\times_{k'}X_{r,k'},\Delta_{\mathbb P(E_{k'})/k'}(M')}$ est un module obtenu de l'id\'eal diagonal de l'anneau $\O_{X_{1}\times_k\cdots\times_{k}X_{r},\Delta_{\mathbb P(E)/k}(M)}$ par extension scalaire par rapport au changement de base $\spec k'\rightarrow\spec k$.

Soient $\mathcal I$ le faisceau d'id\'eaux de $\O_{X_1\times_k\cdots\times_k X_r}$ correspondant au sous-sch\'ema ferm\'e $X_1\cap\cdots\cap X_r$ via le morphisme diagonal, et $\mathcal I'$ le faisceau d'id\'eaux de $\O_{X_{1,k'}\times_{k'}\cdots\times_{k'} X_{r,k'}}$ correspondant au sous-sch\'ema ferm\'e $X_{1,k'}\cap\cdots\cap X_{r,k'}$ via le morphisme diagonal (voir \S1.2.4 pour la d\'efinition). On d\'esigne par $\Delta$ les morphismes diagonaux d\'efinis ci-dessus pour simplifier. De plus, soient $\eta$ le point g\'en\'erique du $\Delta(Y)$, et $\eta'$ le point g\'en\'erique de $\Delta(Y')$. Par l'argument ci-dessus, on a
\[\mathcal I_\eta\O_{X_{1,k'}\times_{k'}\cdots\times_{k'}X_{r,k'},\Delta(Y')}=\mathcal I'_{\eta'}\O_{X_{1,k'}\times_{k'}\cdots\times_{k'}X_{r,k'},\Delta(Y')}\]
 comme id\'eaux de l'anneau $\O_{X_{1,k'}\times_{k'}\cdots\times_{k'}X_{r,k'},\Delta(Y')}$.

On peut confirmer que $\O_{X_{1,k'}\times_{k'}\cdots\times_{k'}X_{r,k'},\Delta(Y')}$ est un $\O_{X_1\times_k\cdots\times_k X_r,\Delta(Y)}$-module plat, car le morphisme canonique
\begin{equation}\label{etale embedding}
\O_{X_1\times_k\cdots\times_k X_r,\Delta(Y)}\hookrightarrow\O_{X_{1,k'}\times_{k'}\cdots\times_{k'}X_{r,k'},\Delta(Y')}
 \end{equation}
  est une composition d'une extension de corps et une localisation. De plus, comme l'extension $k'/k$ est s\'eparable, le morphisme \eqref{etale embedding} est \'etale.

  On d\'esigne par $\kappa(Y)$ le corps r\'esiduel du point g\'en\'erique de $\Delta(Y)$ vu comme point sch\'ematique de $X_1\times_k\cdots\times_k X_r$, et par $\kappa(Y')$ le corps r\'esiduel du point g\'en\'erique de $\Delta(Y')$ vu comme point sch\'ematique de $X_{1,k'}\times_{k'}\cdots\times_{k'}X_{r,k'}$. Comme le morphisme \eqref{etale embedding} est \'etale, d'apr\`es \cite[Proposition 3.2(e)]{Milne}, on a le diagramme cart\'esien suivant:
  \[\xymatrix{\coprod\limits_{Y'\in\mathcal C(Y_{k'})}\spec\kappa(Y')\ar[r]\ar[d]\ar@{}|-{\square}[dr]&\spec\kappa(Y)\ar[d]\\
  \spec\left(\O_{X_1\times_k\cdots\times_k X_r,\Delta(Y)}\otimes_{k}k'\right)\ar[r]\ar[d]\ar@{}|-{\square}[dr]&\spec\O_{X_1\times_k\cdots\times_k X_r,\Delta(Y)}\ar[d]\\
  \spec k'\ar[r]&\spec k.}\]
  Donc on obtient l'\'egalit\'e
  \begin{equation}\label{sum of residue field}
\sum_{Y'\in\mathcal C(Y_{k'})}[\kappa(Y'):\kappa(Y)]=[k':k],
\end{equation}
car le changement de base est \'etale.

D'apr\`es \cite[Chap. II, n$^\circ$ 5, f, coro. 2]{SamuelLocAlg}, on obtient
\begin{equation*}
[k':k]e_{\mathcal I_\eta,\O_{X_1\times_k\cdots\times_k X_r,\Delta(Y)}}=\sum_{Y'\in\mathcal C(Y_{k'})}[\kappa(Y'):\kappa(Y)]e_{\mathcal I'_{\eta'},\O_{X_{1,k'}\times_{k'}\cdots\times_{k'}X_{r,k'},\Delta(Y')}}.
\end{equation*}
On en d\'eduit
\begin{eqnarray}\label{sum of multiplicity with residue field}
  & &[k':k]i(Y;X_1\intersect X_r;\mathbb P(E))\\
  &=&\sum_{Y'\in\mathcal C(Y_{k'})}[\kappa(Y'):\kappa(Y)]i(Y';X_{1,k'}\intersect X_{r,k'};\mathbb P(E_{k'}))\nonumber
\end{eqnarray}
par la d\'efinition de la multiplicit\'e d'intersection (voir \S\ref{mult of a sub-scheme} pour la d\'efinition).

 Comme $X_{i,k'}$ est $\gal(k'/k)$-invariant pour tout $i=1,\ldots,r$, et tous les \'el\'ements dans $\mathcal C(Y_{k'})$ sont dans la m\^eme orbite galoisienne d'apr\`es \cite[Proposition A.14]{Mustata-notes} car l'extension $k'/k$ est galoisienne, la fontion $i(\cdot\;;X_{1,k'}\intersect X_{r,k'};\mathbb P(E_{k'}))$ est constante sur $\mathcal C(Y_{k'})$. Donc d'apr\`es les \'egalit\'es \eqref{sum of residue field} et \eqref{sum of multiplicity with residue field}, on a l'assertion.

\end{proof}

\begin{prop}\label{mult of point under base change}
  Soient $X$ un sous-sch\'ema ferm\'e de $\mathbb P(E)$, et $Y$ un sous-sch\'ema ferm\'e int\`egre de $X$. Soit $k'/k$ une extension galoisienne finie de corps. Pour tout $Y'\in \mathcal C(Y_{k'})$, on a
   \[\mu_Y(X)=\mu_{Y'}(X_{k'}).\]
\end{prop}
\begin{proof}
 On utilise la m\'ethode similaire \`a la d\'emonstration de la proposition \ref{basechangemultinter}. D'apr\`es \cite[Proposition 3.2.7]{LiuQing}, pour toute composante irr\'eductible $Y'\in\mathcal C(Y_{k'})$, on a $\dim(Y')=\dim (Y)$. Toutes les $Y'\in\mathcal C(Y_{k'})$ sont isomorphes comme $k'$-sch\'emas d'apr\`es \cite[Proposition A.14]{Mustata-notes}. De plus, on peut confirmer que $\O_{X_{k'},Y'}$ est un $\O_{X,Y}$-module plat, car le morphisme canonique
\begin{equation}\label{etale embedding2}
\O_{X,Y}\hookrightarrow\O_{X_{k'},Y'}
 \end{equation}
  est une composition d'une extension de corps et une localisation. De plus, comme $k'/k$ est une extension s\'eparable, le morphisme \eqref{etale embedding2} est \'etale.

  On d\'esigne par $\kappa(Y)$ le corps r\'esiduel de l'anneau $\O_{X,Y}$, et par $\kappa(Y')$ le corps r\'esiduel de l'anneau $\O_{X_{k'},Y'}$. D'apr\`es \cite[Proposition 3.2(e)]{Milne}, on a le diagramme cart\'esien suivant:
  \[\xymatrix{\coprod\limits_{Y'\in\mathcal C(Y_{k'})}\spec\kappa(Y')\ar[r]\ar[d]\ar@{}|-{\square}[dr]&\spec\kappa(Y)\ar[d]\\
  \spec\left(\O_{X,Y}\otimes_{k}k'\right)\ar[r]\ar[d]\ar@{}|-{\square}[dr]&\spec\O_{X,Y}\ar[d]\\
  \spec k'\ar[r]&\spec k.}\]
  Donc on a l'\'egalit\'e
  \begin{equation}\label{sum of residue field2}
\sum_{Y'\in\mathcal C(Y_{k'})}[\kappa(Y'):\kappa(Y)]=[k':k],
\end{equation}
car le changement de base est \'etale.

Soient $\sm_{\O_{X,Y}}$ l'id\'eal maximal de l'anneau $\O_{X,Y}$, et $\sm_{X_{k'},Y'}$ l'id\'eal maximal de l'anneau $\O_{X_{k'},Y'}$. Alors on a $\sm_{X_{k'},Y'}=\O_{X_{k'},Y'}\sm_{\O_{X,Y}}$ comme le morphisme \eqref{etale embedding2} est \'etale. D'apr\`es \cite[Chap. II, n$^\circ$ 5, f, coro. 2]{SamuelLocAlg}, on a
\begin{equation*}
[k':k]e_{\sm_{X,Y},\O_{X,Y}}=\sum_{Y'\in\mathcal C(Y_{k'})}[\kappa(Y'):\kappa(Y)]e_{\sm_{X_{k'},Y'},\O_{X_{k'},Y'}}.
\end{equation*}
On en d\'eduit
\begin{equation}\label{sum of multiplicity with residue field2}
  [k':k]\mu_Y(X)=\sum_{Y'\in\mathcal C(Y_{k'})}[\kappa(Y'):\kappa(Y)]\mu_{Y'}(X_{k'}).
\end{equation}

 Comme $X_{k'}$ est $\gal(k'/k)$-invariant pour tout $i=1,\ldots,r$, et tous les \'el\'ements dans $\mathcal C(Y_{k'})$ sont dans la m\^eme orbite galoisienne d'apr\`es \cite[Proposition A.14]{Mustata-notes} car l'extension $k'/k$ est galoisienne, alors la fontion $\mu_{(\ndot)}(X_{k'})$ est constante sur $\mathcal C(Y_{k'})$. Donc d'apr\`es les \'egalit\'es \eqref{sum of residue field2} et \eqref{sum of multiplicity with residue field2}, on a l'assertion.

\end{proof}

\subsubsection*{Comparaision des multiplicit\'es}
On appelle \textit{sous-sch\'ema ferm\'e $k$-lin\'eaire de $\mathbb P(E)$} (ou \textit{sous-sch\'ema lin\'eaire ferm\'e de $\mathbb P(E)$} pour simplifier s'il n'a y pas d'ambigu\"it\'e sur le corps de base) de dimension $d$ toute intersection compl\`ete de $n-d$ $k$-hyperplans de $\mathbb P(E)$. On peut d\'emontrer qu'il est un sous-sch\'ema ferm\'e int\`egre de $\mathbb P(E)$ de degr\'e $1$ par rapport au fibr\'e universel.
\begin{defi}[Cylindre]\label{def of cylindre}
Soient $X$ un sous-sch\'ema ferm\'e de $\mathbb P(E)$ de dimension pure $d$, o\`u $d<n=\rg_k(E)-1$, et $P$ un point dans $X(k)$. Soit $L$ un sous-sch\'ema ferm\'e de $\mathbb P(E)$. On dit que $X$ et $L$ \textit{s'intersectent seulement au voisinage de $P$} si $L$ contient $P$ et si toute composante irr\'eductible de $X\cap L$ passant par $P$ se r\'eduit \`a $\{P\}$. Dans le reste de la d\'efinition, on fixe un sous-sch\'ema $k$-lin\'eaire ferm\'e $L$ de $\mathbb P(E)$ tel que $X$ et $L$ s'intersectent seulement au voisinage de $P$.

 Dans la suite, on d\'efinit une application rationnelle $\phi:\mathbb P(E)\times_k\mathbb P(E)\dashrightarrow \mathbb P(E)$. Le point $P\in\mathbb P(E)(k)$ correspond \`a un homomorphisme surjectif $E\rightarrow\O_{\mathbb P(E)}(1)|_P$. Soit $H_P=\ker(E\rightarrow\O_{\mathbb P(E)}(1)|_P)$. On fixe une application $k$-lin\'eaire injective $\psi:k\rightarrow E$. On d\'esigne par $
U_\psi=\mathbb P(E)\smallsetminus V(\psi)$, o\`u $V(\psi)$ est l'hyperplan d\'efini par l'application $k$-lin\'eaire $\psi$. On suppose que $V(\psi)$ ne contient ni le point $P$ ni le point g\'en\'erique de $X$.

Si $R$ est une $k$-alg\`ebre, alors $U_\psi(R)$ est l'ensemble des applications $R$-lin\'eaires $f:E\otimes_kR\rightarrow R$ telles que la composition des morphismes
\[\begin{CD}
  R@>\psi\otimes\Id>>E\otimes_kR@>f>>R
\end{CD}\]
soit l'application d'identit\'e de $R$. Cet ensemble est en bijection fonctorielle (en $R$) \`a l'ensemble des applications $R$-lin\'eaires de $H_P$ vers $R$. Ainsi on peut identifier le $k$-sch\'ema $U_\psi$ \`a l'espace affine $\mathbb A(H_P)$. La coordonn\'ee affine du point $P\in U_\psi$ dans $U_\psi(H_P)$ est $0\in H_P^\vee$.

L'espace affine $\mathbb A(H_P)$ est un sch\'ema en groupes en consid\'erant la loi d'addition canonique $\phi$
\[\phi:\mathbb A(H_P)\times_k\mathbb A(H_P)\rightarrow \mathbb A(H_P)\]
qui envoie tout point $(a,b)$ sur $a+b$.

L'adh\'erence Zariski $Y$ de $\phi(X\times_kL)$ dans $\mathbb P(E)$, qui est de dimension $m+d$ (voir la remarque \ref{the dimension of cylinder} ci-dessous pour une d\'emonstration), est appel\'ee le \textit{cylindre} passant par $X$ de la direction $L$ relativement \`a  $P$. On remarque que la classe rationnelle de $\phi$ et donc le cylindre ne d\'epend pas du choix de $\psi$.
\end{defi}
\begin{rema}\label{the dimension of cylinder}
On d\'emontre que la dimension du cylindre dans la d\'efinition \ref{def of cylindre} est $m+d$. Avec toutes les notations dans la d\'efinition \ref{def of cylindre}, comme $\dim(X\times_kL)=m+d$, on a $\dim (Y)\leqslant m+d$ (cf. \cite[Corollary 3.3.14]{LiuQing}).

Pour l'in\'egalit\'e inverse, on prend un sous-sch\'ema $k$-lin\'eaire ferm\'e $L'$ de $\mathbb P(E)$ de dimension $n-m$ qui intersecte $L$ dans $\mathbb P(E)$ en le point $\{P\}$ seulement. Le morphisme $\phi|_{(U_\psi\cap L)\times_k( U_\psi\cap L')}:(U_\psi\cap L)\times_k(U_\psi\cap L')\rightarrow U_\psi$ est un isomorphisme de sch\'emas. Alors on peut construire un $\overline k$-morphisme $\theta:X_{\overline k}\rightarrow L'_{\overline k}$, tel que $\dim(\theta(X))=d$ et l'image inverse de tout $k$-point de $\theta(X)$ par rapport \`a $\theta$ est un ensemble fini. Alors on peut prendre un sous-ensemble $X'$ de $X_{\overline k}$ de dimension $d$ tel que $\theta:X'\rightarrow L'_{\overline k}$ soit une bijection. Donc le morphisme $\phi_{\overline k}|_{(X'\cap U_{\psi})\times_{k} (L\cap U_{\psi})}$ est une immersion, alors on a $\dim(\phi(X'\times_{\overline k}L_{\overline k}))=m+d$. De plus, on a $\phi(X'\times_{\overline k}L_{\overline k})\subseteq Y_{k'}$ par d\'efinition.

  On a d\'emontr\'e que $Y_{\overline k}$ contient un sous-ensemble de dimension $m+d$. Comme $\dim(Y)=\dim(Y_{\overline k})$ d'apr\`es \cite[Proposition 3.2.7]{LiuQing}, on obtient l'in\'egalit\'e $\dim (Y)\geqslant m+d$, ce qui termine la d\'emonstration.
\end{rema}
Avec les notations ci-dessus, on a la proposition suivante:
\begin{prop}\label{cylindre}
Soit $U$ un sous-sch\'ema ferm\'e int\`egre de $\mathbb P(E)$ tel que $U^{\mathrm{reg}}(k)\neq\emptyset$. Soit $\dim(U)<m<n+dim(U)$ un entier. On fixe un point $P\in U^{\mathrm{reg}}(k)$. Alors il existe un cylindre $U_1$ de dimension $n+\dim(U)-m$ dont la direction est d\'efinie par un sous-sch\'ema $k$-lin\'eaire ferm\'e $L$ de $\mathbb P(E)$ de dimension $n-m$ passant par $P$ tel que, pour tout sous-sch\'ema ferm\'e $V$ de dimension pure $m$ de $\mathbb P(E)$ qui contient $U$, si $L$ intersecte $V$ proprement en le point $P$, alors le cylindre $U_1$ intersecte $V$ proprement en $U$. De plus, on a
\begin{equation*}
  \mu_U(V)=i(U;U_1\cdot V;\mathbb P(E))
\end{equation*}
et
\[\mu_Q(V)=\mu_U(V)\]
pour tout $Q\in U^{\mathrm{reg}}(k)$. Voir \S\ref{multiplicity of local ring} pour la notation de $\mu_U(V)$.
\end{prop}
 On revoie les lecteurs au deuxi\`eme paragraphe de \cite[Chap. II \S6, n$^\circ$ 2, b)]{Samuel} pour une d\'emonstration de la proposition \ref{cylindre}. L'auteur de \cite{Samuel} a implicitement utilis\'e la condition $U^{\mathrm{reg}}(k)\neq\emptyset$ dans la d\'emonstration sans la pr\'eciser dans l'\'enonc\'e.
 \begin{rema}
   Soient $X_1,\ldots, X_r$ des sous-sch\'ema ferm\'es de la m\^eme dimension pure de $\mathbb P(E)$, et $U$ un sous-sch\'ema ferm\'e int\`egre de $X_1,\ldots,X_r$. Par la m\'ethode dans \cite{Gabber_Liu_Lorenzini1} et \cite{Gabber_Liu_Lorenzini2}, on peut construire un sch\'ema $U_1$ d'intersection compl\`ete qui intersecte tous les $X_1,\ldots,X_r$ proprement en la composante irr\'eductible $U$, tel que l'\'egalit\'e
   \[\mu_U(X_i)=i(U;U_1\cdot X_i;\mathbb P(E))\]
   soit v\'erifi\'ee pour tout $i=1,\ldots,r$ d'apr\`es le lemme d'\'evitement (en anglais c'est "avoidance lemma"). Compar\'e avec la proposition \ref{cylindre}, on n'a pas besoin de supposer que $U^{\mathrm{reg}}(k)\neq\emptyset$ et qu'il existe le sous-sch\'ema $k$-lin\'eaire $L$ satisfaisant la condition dans la proposition \ref{cylindre}. Si on admet l'assertion, on peut montrer le th\'eor\`eme \ref{chongshu2} sans la condition que $k$ est un corps parfait.
 \end{rema}
\begin{defi}
Soit $X$ un sch\'ema. On dit qu'une propri\'et\'e d\'epandant d'un point de $X$ est \textit{vraie pour presque tout point de $X$} s'il existe un sous-ensemble dense $U$ de $X$, tel que cette propri\'et\'e soit vraie pour tout point dans $U$.
\end{defi}
Si le sch\'ema $X$ est irr\'eductible, $X^{\mathrm{reg}}$ est dense dans $X$ si $X^{\mathrm{reg}}\neq\emptyset$. On a le corollaire de la proposition \ref{cylindre} suivante.
\begin{coro}\label{sub}
Soit $X$ un sous-sch\'ema ferm\'e de $\mathbb P(E)$. Soient $Y$ et $Z$ deux sous-sch\'emas ferm\'es int\`egres de $X$, o\`u $Z\subseteq Y$ et $Z^{\mathrm{reg}}(k)\neq\emptyset$. Alors on a $\mu_Y(X)\leqslant\mu_Z(X)$. De plus, pour preque tout point $P$ de $Y$, on a $\mu_P(X)=\mu_Y(X)$.
\end{coro}
On revoie les lecteurs \`a \cite[Chap. II \S6, n$^\circ$ 2, c)]{Samuel} pour une d\'emonstration du corollaire \ref{sub}.

On comparera la multiplicit\'e d'intersection d'une famille de sch\'emas et un produit de  multiplicit\'es de cette composante irr\'eductible dans cette famille de sch\'emas. Dans \cite[Chap. II \S6, n$^\circ$ 2, e)]{Samuel}, l'auteur de \cite{Samuel} a d\'emontr\'e la proposition \ref{inqmult}. Mais dans la d\'emonstration, l'auteur de \cite{Samuel} a implicitement utilis\'e la condition que cette composante irr\'eductible est g\'eom\'etriquement int\`egre sans la pr\'eciser dans l'\'enonc\'e. Ici on n'a pas besoin de supposer cette condition, et on peut d\'emontrer le cas o\`u le corps de base est parfait.

Pour cela, on introduira un lemme auxiliaire suivant.
\begin{lemm}\label{regular point induces geometrically integral}
  Soit $X$ un sous-sch\'ema ferm\'e int\`egre de $\mathbb P(E)$. Si l'ensemble $X^{\mathrm{reg}}(k)\neq\emptyset$, alors $X$ est g\'eom\'etriquement int\`egre.
\end{lemm}
\begin{proof}
Il faut montrer que $X$ est g\'eom\'etriquement r\'eduit et g\'eom\'etriquement irr\'eductible.

 D'abord, on va d\'emontrer que $X$ est g\'eom\'etriement irr\'eductible. Soit $\xi\in X^{\mathrm{reg}}(k)$. Pour toute extension de corps $k'/k$, soit $\xi'=\xi\times_{\spec k}\spec k'$. Alors d'apr\`es le crit\`ere jacobien (cf. \cite[Theorem 4.2.19]{LiuQing}), on a
 \[\mu_{\xi'}(X_{k'})=\mu_\xi(X)=1,\]
  comme le rang de la matrice jacobien en un point rationnel est invariant sous l'extension de corps. De plus, si l'extension $k'/k$ est galoisienne, le point $\xi'$ est $\gal(k'/k)$-invariant. Donc pour toute composante irr\'eductible $X'\in\mathcal C(X_{k'})$, on a $\xi'\in X'$.

  D'apr\`es la proposition \ref{sumofmult}, pour toute extension galoisienne $k'/k$, on a l'\'egalit\'e
  \[\sum_{X'\in\mathcal C(X_{k'})}\ell_{\O_{X_{k'},X'}}(\O_{X_{k'},X'})\mu_{\xi'}(X')=\mu_{\xi'}(X_{k'})=1.\]
  Donc on obtient $\#\mathcal C(X_{k'})=1$ et $\ell_{\O_{X_{k'},X'}}(\O_{X_{k'},X'})=1$ pour le $X'\in\mathcal C(X_{k'})$. L'assertion $\#\mathcal C(X_{k'})=1$ signifie que $X_{k'}$ est irr\'eductible. Donc $X$ est g\'eom\'etriquement irr\'eductible.

Dans la suite, on va d\'emontrer que $X$ est g\'eom\'etriement r\'eduit. Si l'extension $k'/k$ est s\'eparable, alors d'apr\`es \cite[Corollary 3.2.14]{LiuQing}, le sch\'ema $X$ est g\'eom\'etriquement r\'eduit.

Si $k'/k$ n'est pas s\'eparable, alors le corps $k$ n'est pas parfait. On suppose que la caract\'eristique de $k$ est $p$. Dans ce cas-l\`a, on peut diviser l'extension \`a une composition d'une extension s\'eparable et une extension purement ins\'eparable. Pour la partie purement ins\'eparable, on peut la diviser \`a une composition des extensions purement ins\'eparable de degr\'e $p$. Alors on a besoin de montre que si $k'/k$ est une extension purement ins\'eparable avec $[k':k]=p$, le sch\'ema $X_{k'}$ est r\'eduit. Comme la question est locale, alors on peut supposer que $X$ est affine. Soit $X=\spec A$, o\`u $A$ est un anneau contenant $k$.

Comme $X$ admet un point $k$-rationnel r\'egulier, alors on prend $\xi\in X^{\mathrm{reg}}(k)$, et on d\'esigne par $\sm_\xi$ l'id\'eal maximal de l'anneau $\O_{X,\xi}$. Alors on a $\widehat{A}_{\sm_\xi}=\widehat{\O}_{X,\xi}\cong k[[T_1,\ldots,T_d]]$ (cf. \cite[(28.J)]{Matsumura1}), o\`u $d=\dim(X)$. Soit $\xi'=\xi\times_{\spec k}\spec k'$, alors on a $\widehat{\O}_{X_{k'},\xi'}\cong k'[[T_1,\ldots,T_d]]$ car $\xi'$ est r\'egulier dans $X_{k'}$. Donc on a le diagramme commutatif suivant:
\[\xymatrix{\relax A\ar@{^{(}->}[d]\ar@{^{(}->}[r]&k[[T_1,\ldots,T_d]]\ar@{^{(}->}[d]\\
A\otimes_kk'\ar@{^{(}->}[r]&k'[[T_1,\ldots,T_d]].}\]

L'anneau $k'[[T_1,\ldots,T_d]]$ est int\`egre, alors l'anneau $A\otimes_kk'$ est int\`egre aussi, qui doit \^etre r\'eduit. Donc on obtient que $X$ est g\'eom\'etriquement r\'eduit. D'o\`u on a le r\'esultat.
\end{proof}
\begin{rema}
  La d\'emonstration du lemme \ref{regular point induces geometrically integral} est similaire \`a celle de \cite[Lemma 10.1]{Poonen_2009_CT}, mais la condition dans le lemme \ref{regular point induces geometrically integral} est plus faible.
\end{rema}
\begin{prop}\label{inqmult}
   On suppose que $k$ est un corps parfait. Soient $X_1,\ldots,X_r$ des sous-sch\'emas ferm\'es de dimension pure de $\mathbb P(E)$ et $M\in\mathcal C(X_1\intersect X_r)$. Alors on a
    \begin{equation*}
    i(M;X_1\intersect X_r;\mathbb P(E))\geqslant\prod\limits_{i=1}^r\mu_M(X_i).
  \end{equation*}
\end{prop}
\begin{proof}
D'abord, on suppose $M^{\mathrm{reg}}(k)\neq\emptyset$. Dans ce cas-l\`a, d'apr\`es le lemme \ref{regular point induces geometrically integral}, le sch\'emas $M$ est g\'eom\'etriquement int\`egre. D'o\`u l'on obtient que le sch\'ema $M^{\times_kr}$ est g\'eom\'etriquement int\`egre aussi par \cite[(4.6.5) (ii)]{EGAIV_2}.

   La multiplicit\'e d'intersection $i(M;X_1\intersect X_r;\mathbb P(E))$ est la multiplicit\'e d'un id\'eal de l'anneau local $\O_{X_1\times_k\cdots\times_k X_r,\Delta(M)}$ qui est contenu dans l'id\'eal maximal de $\O_{X_1\times_k\cdots\times_k X_r,\Delta(M)}$. D'apr\`es l'in\'egalit\'e \eqref{smaller ideal with bigger multiplicity}, on obtient
   \[i(M;X_1\intersect X_r;\mathbb P(E))\geqslant\mu_{\Delta(M)}(X_1\times_k\cdots\times_k X_r).\]
    De plus, le sch\'ema $\Delta(M)$ est g\'eom\'etriquement int\`egre et il admet un point r\'egulier $k$-rationnel. D'apr\`es le fait que $\Delta(M)\subseteq M^{\times_kr}$, on obtient
  \begin{equation*}
    \mu_{\Delta(M)}(X_1\times_k\cdots\times_k X_r)\geqslant\mu_{M^{\times_kr}}(X_1\times_k\cdots\times_k X_r)
  \end{equation*}
  compte tenu du corollaire \ref{sub}.

  Soient $U_1$ et $U_2$ deux sous-sch\'emas ferm\'es g\'eom\'etriquement int\`egres de $Y_1$ et $Y_2$ respectivement, o\`u $Y_1$ et $Y_2$ sont deux sous-sch\'emas ferm\'es de $\mathbb P(E)$. D'apr\`es \cite[(4.6.5) (ii)]{EGAIV_2}, le sch\'ema $U_1\times_k U_2$ est g\'eom\'etriquement int\`egre. Dans ce cas-l\`a, le sch\'ema $U_1\times_k U_2$ est un sous-sch\'ema ferm\'e int\`egre de $Y_1\times_k Y_2$, d'o\`u $\O_{Y_1\times_k Y_2,U_1\times_k U_2}\cong\O_{Y_1,U_1}\otimes_k\O_{Y_2,U_2}$. D'apr\`es \cite[Chap. VI, n$^\circ$ 1, d, prop. 1]{SamuelLocAlg}, on en d\'eduit
\begin{equation*}
  \mu_{U_1\times_kU_2}(Y_1\times_k Y_2)=\mu_{U_1}(Y_1)\mu_{U_2}(Y_2).
\end{equation*}
 Alors on a
\begin{eqnarray*}
  \mu_{M^{\times_kr}}(X_1\times_k\cdots\times_k X_r)&=&\mu_M(X_1)\cdot\mu_{M^{\times_k(r-1)}}(X_2\times_k\cdots\times_k X_r)\\
  &=&\cdots\\
  &=&\prod\limits_{i=1}^r\mu_M(X_i),
\end{eqnarray*}
qui d\'emontre l'assertion.

  Dans la suite, on va d\'emontrer le cas o\`u $k$ est un corps parfait et $M\in\mathcal C(X_1\intersect X_r)$. Soit $k'/k$ une extension galoisienne finie de corps telle que pour toute composante irr\'eductible $M'\in\mathcal C(M_{k'})$, $M'$ contienne au moins un $k'$-point r\'egulier. D'apr\`es le lemme \ref{regular point induces geometrically integral}, toute $M'\in\mathcal C(M_{k'})$ est g\'eom\'etriquement int\`egre. D'apr\`es l'argument ci-dessus, si on fixe une $M'\in\mathcal C(M_{k'})\subseteq\mathcal C(X_{1,k'}\intersect X_{r,k'})$ (par le lemme \ref{irreducible component change after extension of fields}), on a
      \begin{equation*}
    i(M';X_{1,k'}\intersect X_{r,k'};\mathbb P(E_{k'}))\geqslant\prod\limits_{i=1}^r\mu_{M'}(X_{i,k'}).
  \end{equation*}
  D'apr\`es la proposition \ref{basechangemultinter}, on a
  \[i(M;X_1\intersect X_r;\mathbb P(E))=i(M';X_{1,k'}\intersect X_{r,k'};\mathbb P(E_{k'})).\]
  Par la proposition \ref{mult of point under base change}, on a
  \[\mu_M(X_i)=\mu_{M'}(X_{i,k'}).\]
  Alors on a l'assertion.
\end{proof}

\subsubsection*{Comptage des objets sur un corps fini}
Soient $k$ un corps et $V$ un espace $k$-vectoriel de rang fini. On d\'esigne par $\Gr(r,V^\vee)$ la grassmannienne qui classifie les sous-espaces vectoriels de dimension $r$ de $V$. Soit $k'/k$ une extension de corps, on d\'esigne par $\Gr(r,V^\vee)(k')$ l'ensemble des $k$-points \`a valeurs dans le corps $k'$ de $\Gr(r,V^\vee)$. On d\'esigne par $\Gr_k(r,n)$ la grassmannienne $\Gr(r,(k^n)^\vee)$, ou par $\Gr(r,n)$ s'il n'y a pas d'ambigu\"it\'e sur le corps de base $k$. En particulier, on a $\Gr_k(n-1,n)\cong\mathbb P_k^{n-1}$.

\begin{lemm}\label{grassmanne}
  Avec les notations ci-dessus, soit $\f_q$ le corps fini de cardinal $q$. Alors on a
  \begin{equation*}
    \#\Gr_{\f_q}(r,n)(\f_q)=\frac{\prod\limits_{t=1}^n(q^{t-1}+q^{t-2}+\cdots+1)}{\prod\limits_{t=1}^r(q^{t-1}+q^{t-2}+\cdots+1)\cdot\prod\limits_{t=1}^{n-r}(q^{t-1}+q^{t-2}+\cdots+1)}.
  \end{equation*}
  En particulier, on a
\[\mathbb P^n_{\f_q}(\f_q)=q^n+\cdots+1.\]
\end{lemm}
On revoie les lecteurs \`a \cite[Proposition 1.7.2]{Stanley} pour une d\'emonstration du lemme \ref{grassmanne}.

Soient $k'/k$ une extension de corps, $E$ un espace $k$-vectoriel de rang fini, et $\phi:X\hookrightarrow\mathbb P(E_{k'})$ une immersion ferm\'ee. On a le diagramme commutatif suivant:
\[\xymatrix{X\ar@{^{(}->}^{\phi}[r]&\mathbb P(E_{k'})\ar[r]^\pi\ar[d]\ar@{}|-{\square}[dr]&\mathbb P(E)\ar[d]\\
&\spec k'\ar[r]&\spec k.}\]
\begin{defi}\label{descent of k-points}
On d\'esigne par $X_{\phi}(k)$ le sous-ensemble de $X(k')$ des $\xi\in X(k')$ (consid\'er\'es comme des $k'$-morphismes de $\spec k'$ dans $X$) dont la composition avec le morphisme canonique $X\rightarrow\mathbb P(E)$ donne un $k$-point de $\mathbb P(E)$ \`a valeurs dans $k'$ qui provient d'un point $k$-rationnel de $\mathbb P(E)$. Autrement dit, on d\'efinit $X_\phi(k)=X(k')\cap\pi^{-1}\left(\mathbb P(E)(k)\right)$. S'il n'y a pas d'ambigu\"it\'e sur l'immersion $\phi$, on d\'esigne par $X(k)$ l'ensemble $X_{\phi}(k)$ pour simplifier.
\end{defi}
Lorsque $k$ est un corps fini , on a un r\'esultat comme ci-dessous pour estimer le cardinal de l'ensemble $X_{\phi}(k)$ lorsque $X$ est de dimension pure.
\begin{prop}\label{lineaire}
  Soient $k/\f_q$ une extension de corps, $E$ un espace $k$-vectoriel de rang fini, et $\phi:X\hookrightarrow\mathbb P(E)$  une immersion ferm\'ee. On suppsose que $X$ est de dimension pure $d$. Alors
  \[\#X_{\phi}(\f_q)\leqslant\deg(X)\#\mathbb P^d_{\f_q}(\f_q).\]
\end{prop}
On revoie les lecteurs \`a l'argument dans la page 236 de \cite{Mazur_1975_ag}. La proposition \ref{lineaire} est une cons\'equence directe de cet argument.

Soient $k$ un corps, et $X_1,\ldots,X_r$ des $k$-sch\'emas tels que $\bigcap_{i=1}^rX_i(k)\neq\emptyset$. Si $P\in \bigcap_{i=1}^rX_i(k)$, et toute composante irr\'eductible de l'intersection de $X_1\cap\cdots\cap X_r$ passant par $P$ se r\'eduit \`a $\{P\}$, on dit que \textit{$X_1,\ldots,X_r$ s'intersectent seulement au voisinage de $P$}.

La proposition suivante est utilis\'ee pour d\'eterminer s'il existe un sous-sch\'ema $k$-lin\'eaire ferm\'e de $\mathbb P(E)$ qui intersecte une suite des sch\'emas de dimension pure fix\'es en un point $k$-rationnel seulement au voisinage de ce point.

\begin{prop}\label{numofproper}
  Soient $U_1,\ldots,U_r$ des sous-sch\'emas ferm\'es de dimension pure de $\mathbb P(E)$. On suppose que $\bigcap_{i=1}^rU_i(k)\neq\emptyset$ et $\dim(U_i)=d<n=\rg_k(E)-1$ pour tout $i=1,\ldots,r$. Soit $P\in \bigcap_{i=1}^rU_i(k)$. Si l'in\'egalit\'e
  \[\#k\geqslant\deg (U_1)+\cdots+\deg (U_r)\]
  est v\'erifi\'ee, alors il existe au moins un sous-sch\'ema ferm\'e $k$-lin\'eaire de $\mathbb P(E)$ de dimension plus petite ou \'egale \`a $n-d$ qui intersecte tout $U_i$ seulement au voisinage de $P$.
\end{prop}
\begin{proof}
S'il existe un sous-sch\'ema $k$-lin\'eaire ferm\'e $L$ de $\mathbb P(E)$ de dimension $n-d$ qui intersecte tous les $U_1,\ldots,U_r$ proprement en le point $P$, alors pour tout sous-sch\'ema $k$-lin\'eaire ferm\'e de $\mathbb P(E)$ passant par $P$ contenu dans $L$, il intersecte $U_1,\ldots,U_r$ seulement au voisinage de $P$. Donc on a besoin de prouver qu'il existe un sous-sch\'ema $k$-lin\'eaire ferm\'e $L'$ de $\mathbb P(E)$ de dimension $n-d$ tel que $\{P\}$ soit une composante propre de l'intersection $L'\cdot U_1\intersect U_r$ dans $\mathbb P(E)$.

On d\'esigne par $\mathscr H_P$ l'ensemble des $k$-hyperplans projectifs passant par le point $P$, alors on a $\mathscr H_P=\Gr(n-1,E^\vee)(k)$. D'abord, on d\'emontrera que l'on peut trouver une $H_1\in\mathscr H_P$ qui intersecte tous les $U_i$ proprement. Pour un $U_i$ fix\'e, ses composantes irr\'eductibles sont contenues dans au plus $\deg(U_i)$ sous-sch\'emas $k$-lin\'eaires ferm\'es de $\mathbb P(E)$ de dimension $d$. De plus, pour un sous-sch\'ema $k$-lin\'eaire ferm\'e de $\mathbb P(E)$ de dimension $d$ fix\'e, il existe $\#\Gr(n-d-1,n-d)(k)$ hyperplans qui contiennent ce sous-sch\'ema $k$-lin\'eaire ferm\'e de $\mathbb P(E)$. Si $k$ est un corps fini, $\#\Gr(m,n)(k)$ \'etant calcul\'e dans le lemme \ref{grassmanne}, alors on peut comfirmer que l'on a l'in\'egalit\'e
\begin{eqnarray*}
  \#\mathscr H_P&=&\#\Gr(n-1,n)(k)\\
  &>&\left(\deg (U_1)+\cdots+\deg (U_r)\right)\#\Gr(n-d-1,n-d)(k),
\end{eqnarray*}
 lorsque $\#k\geqslant r\geqslant1$ et $\#k\geqslant2$. Donc il existe toujours un tel hyperplan $H_1$.

 Si $k$ est infini, il toujours existe un hyperplan $H_1\in\mathscr H_P$ qui satisfait que les sch\'emas $U_1,\ldots,U_r, H_1$ s'intersectent proprement en une composante irr\'eductible contenant le point $P$.

Si on a d\'ej\`a trouv\'e des hyperplans $H_1,\ldots,H_{t-1}\in \mathscr H_P$, tels que les sch\'emas $U_i,H_1,\ldots,H_{t-1}$ s'intersectent proprement pour tout $i=1,\ldots,r$, o\`u $1\leqslant t\leqslant d$. D'apr\`es le th\'eor\`eme de B\'ezout (le th\'eor\`eme \ref{bezout}), on obtient qu'il y a au plus $\deg (U_i)$ \'el\'ements dans $\mathcal C(H_1\cdot H_2\intersect H_{t-1}\cdot U_i)$, o\`u tout \'el\'ement est de dimension $d-t+1$. De plus, tout \'el\'ement dans  $\mathcal C(H_1\cdot H_2\intersect H_{t-1}\cdot U_i)$ est contenu dans au plus un sous-sch\'ema $k$-lin\'eaire ferm\'e de $\mathbb P(E)$ de dimension $d-t+1$, o\`u $i=1,\ldots,r$. Si $k$ est un corps fini, d'apr\`es la proposition \ref{grassmanne}, on peut confirmer que l'on a
\begin{eqnarray*}
  & &\#\Gr(n-t,n-t+1)(k)\\
  &>&\left(\deg (U_1)+\cdots+\deg (U_r)\right)\#\Gr(n-d-1,n-d)(k),
\end{eqnarray*}
 lorsque $\#k\geqslant r\geqslant1$, $\#k\geqslant2$ et $t\leqslant d$. Donc on peut trouver un sous-sch\'ema $k$-lin\'eaire ferm\'e de $\mathbb P(E)$ de dimension $n-t$ passant par $P$ contenu dans $H_1\cap\cdots\cap H_{t-1}$, qui intersecte tous les \'el\'ements dans $\mathcal C(H_1\cdot H_2\intersect H_{t-1}\cdot U_i)$ proprement pour tout $i=1,\ldots,r$.

  Tout sous-sch\'ema $k$-lin\'eaire ferm\'e de $\mathbb P(E)$ passant par $P$ contenu dans $H_1\cap\cdots\cap H_{t-1}$ peut \^etre relev\'e \`a un hyperplan dans $\mathscr H_P$. On se rel\`eve ce sous-sch\'ema $k$-lin\'eaire ferm\'e de $\mathbb P(E)$ \`a $H_t\in \mathscr H_P$ telle que $H_1\cap\cdots\cap H_{t-1}\cap H_t$ soit une intersection compl\`ete.

 Si $k$ est infini, il toujours existe un hyperplan $H_t\in\mathscr H_P$ qui satisfait que les sch\'emas projectifs $U_1,\ldots,U_r, H_1,\ldots,H_{t-1},H_{t}$ s'intersectent proprement en une composante irr\'eductible contenant le point $P$.

Donc on peut trouver une suite des \'el\'ements $H_1,H_2,\ldots,H_{d}\in\mathscr H_P$, tels que les sch\'emas $H_1,H_2,\ldots,H_{d},U_i$ s'intersectent proprement en le point $P$ pour tout $i=1,\ldots,r$. Le sous-sch\'ema $k$-lin\'eaire ferm\'e de $\mathbb P(E)$ d\'efini par l'intersection compl\`ete de $H_1,H_2,\ldots,H_{d}$ intersecte tous les $U_i$ proprement en le point $P$, o\`u $i=1,\ldots,r$.
\end{proof}

\section{D\'emonstration du th\'eor\`eme \ref{chongshu2}}
Ce paragraphe est consancr\'e \`a la d\'emonstration du th\'eor\`eme \ref{chongshu2}. Soient $k$ un corps parfait, et $X_1,\ldots,X_r$ des sous-sch\'emas ferm\'es de dimension pure de $\mathbb P(E)$ qui s'intersectent proprement. Pour tout $Y\in\mathcal C(X_1\intersect X_r)$, on construit un arbre d'intersection $\mathscr T_Y$ de niveau $\delta=\max\limits_{i\in\{1,\ldots,r\}}\{\deg(X_i)\}$ dont la racine est $Y$. La strat\'egie consiste en un raisonnement par r\'ecurrence sur la profondeur maximale des arbres d'intersection $\mathscr T_Y$ (voir \S\ref{definition of intersection tree} pour la d\'efinition). Soit $M$ un sommet de l'un des arbres d'intersection $\mathscr T_Y$. On suppose que $M$ satisfait les conditions suivantes: pour tout sommet $Z$ des arbres d'intersection $\{\mathscr T_Y\}_{Y\in\mathcal C(X_1\intersect X_r)}$, si $M$ est un sous-sch\'ema propre de $Z$, alors il existe un descendant de $Z$ qui est une occurence de $M$. Le but de ce paragraphe est de d\'emontrer l'in\'egalit\'e \eqref{no auxillary scheme} ci-dessous:
  \begin{equation*}
  \sum_{Y\in\mathcal C(X_1\intersect X_r)}W_{\mathscr T_Y}(M)i(Y;X_1\intersect X_r;\mathbb P(E))\geqslant \mu_{M}(X_1)\cdots\mu_{M}(X_r).
  \end{equation*}

\begin{defi}\label{Cs}
  Soit $s$ un entier positif. On d\'efinit $\mathcal C_s$ comme l'ensemble des sommets de profondeur $s$ dans les arbres d'intersection $\mathscr T_Y$, o\`u $Y\in\mathcal C(X_1\intersect X_r)$. De plus, on d\'efinit  $\mathcal C_*=\bigcup\limits_{s\geqslant0}\mathcal C_s$.
\end{defi}
 Pour tout entier positif $s$, on d\'efinit un sous-ensemble de $\mathcal C_s$ comme la suite.
\begin{defi}\label{Zs}
  Soit $s$ un entier positif. On d\'efinit $\mathcal{Z}_s$ comme le sous-ensemble de $\mathcal C_s$ des \'el\'ements $N$ qui satisfont la condition suivante: pour tout sommet $Z$ des arbres d'intersection $\{\mathscr T_Y\}_{Y\in\mathcal C(X_1\intersect X_r)}$, si $N$ est un sous-sch\'ema propre de $Z$, alors il existe un descendant de $Z$ qui est une occurrence de $N$. De plus, on d\'efinit $\mathcal Z_*=\bigcup\limits_{s\geqslant0}\mathcal Z_s$.
\end{defi}
Par d\'efinition, on a $\mathcal{Z}_0=\mathcal C_0=\mathcal C(X_1\intersect X_r)$. On d\'emontrera le th\'eor\`eme \ref{chongshu2} pour les sommets dans l'ensemble $\mathcal Z_*$.

L'id\'ee principale de la d\'emonstration du th\'eor\`eme \ref{chongshu2} est comme ci-dessous: si $M\in \mathcal Z_0$, le c\^ot\'e gauche de l'in\'egalit\'e \eqref{no auxillary scheme} est une multiplicit\'e d'intersection en $M$, alors le th\'eor\`eme \ref{chongshu2} d\'ecoule de la proposition \ref{inqmult}. Si $M\in \mathcal Z_*\smallsetminus\mathcal Z_0$, comme $k$ est un corps parfait, la multiplicit\'e d'intersection et la multiplicit\'e de point dans un sch\'ema v\'erifient des proprit\'et\'es d'invariance sous une extension galoisienne finie de corps comme dans les propositions \ref{basechangemultinter} et \ref{mult of point under base change}. D'abord on fixe une extension galoisienne finie de corps $k'/k$ tel que $M^{\mathrm{reg}}(k')\neq\emptyset$ et que le cardinal de $k'$ soit assez grand. Alors on peut construire un $k'$-sch\'ema auxiliaire tel que l'intersection de $X_{1,k'},\ldots,X_{r,k'}$ et ce sch\'ema soit propre en une composante irr\'eductible de $M_{k'}$ (le sch\'ema auxiliaire est en fait un cylindre passant par cette composante irr\'eductible, dont l'existence est assur\'ee
lorsque $k'$ est assez grand, voir la d\'efinition \ref{cylindre} pour la d\'efinition de cylindre). Dans la suite, on d\'emontre le c\^ot\'e gauche de l'in\'egalit\'e \eqref{no auxillary scheme} est plus grand ou \'egal \`a la multiplicit\'e d'intersection de $X_{1,k'},\ldots,X_{r,k'}$ et le $k'$-sch\'ema auxiliaire en cette composante irr\'eductible de $M_{k'}$. Par la comparaison entre la multiplicit\'e de ce produit d'intersection \`a cette composante irr\'eductible de $M_{k'}$ et les multiplicit\'es de $M$ dans $X_1,\ldots,X_r$ et le sch\'ema auxiliaire (la proposition \ref{inqmult}), on obtient le r\'esultat.

\subsection*{D\'emonstration du th\'eor\`eme \ref{chongshu2}}
Dans la d\'emonstration, la composante irr\'eductible $M\in\mathcal Z_*$ est comme dans l'\'enonc\'e du th\'eor\`eme \ref{chongshu2}.

\textbf{\'Etape 1: la profondeur du sommet est z\'ero. - }
Si $M\in \mathcal Z_0=\mathcal C(X_1\intersect X_r)$, alors pour tout $Y\in\mathcal C(X_1\intersect X_r)$, on a $W_{\mathscr T_Y}(M)=0$ ou $1$. Donc l'assertion du th\'eor\`eme \ref{chongshu2} est une cons\'equence directe de la proposition \ref{inqmult}, ce qui montre que la multiplicit\'e d'intersection du produit d'intersection de $X_1\intersect X_r$ en une composante irr\'eductible est plus grande ou \'egale au produit des multiplicit\'es de cette composante dans $X_1,\ldots,X_r$.

\textbf{\'Etape 2: la profondeur du sommet est strictement plus grande que z\'ero. - }
Si $M\in \mathcal Z_*\smallsetminus\mathcal Z_0$ On d\'emontrera l'\'enonc\'e suivant.

 \begin{prop}\label{auxillary scheme}
 Soit $n=\rg_k(E)-1$. Soit $k'/k$ une extension galoisienne finie de corps, telle que \[\#k'\geqslant\delta^{\sum\limits_{i=1}^r\dim(X_i)-(r-1)(n-1)}\]
 et que $M^{\mathrm{reg}}(k')\neq\emptyset$. Alors pour toute composante irr\'eductible $M'\in\mathcal C(M_{k'})$, il existe un cylindre $M^{0}_{k'}$ dans $\mathbb P(E_{k'})$, tel que $M'\in\mathcal C(X_{1,k'}\intersect X_{r,k'}\cdot M^0_{k'})$ et que les sch\'emas $X_{1,k'},\ldots,X_{r,k'},M^0_{k'}$ s'intersectent proprement en la composante $M'$, et
\begin{eqnarray*}
    & &\sum_{Y\in\mathcal C(X_1\intersect X_r)}i(Y;X_1\intersect X_r;\mathbb P(E))W_{\mathscr T_Y}(M)\\
    &\geqslant& i(M';X_{1,k'}\intersect X_{r,k'}\cdot M^0_{k'};\mathbb P(E_{k'})).\nonumber
  \end{eqnarray*}
\end{prop}

Si on admet la proposition \ref{auxillary scheme}, d'apr\`es la proposition \ref{inqmult}, on obtient
  \begin{eqnarray*}
     i(M';X_{1,k'}\intersect X_{r,k'}\cdot M^0_{k'};\mathbb P(E_{k'}))&\geqslant&\mu_{M'}(X_{1,k'})\cdots\mu_{M'}(X_{r,k'})\mu_{M'}(M^0_{k'})\\
     &\geqslant&\mu_{M'}(X_{1,k'})\cdots\mu_{M'}(X_{r,k'})\\
&=&\mu_{M}(X_1)\cdots\mu_{M}(X_r),
  \end{eqnarray*}
o\`u l'\'egalit\'e provient de la proposition \ref{mult of point under base change}. D'o\`u on d\'emontre l'assertion du th\'eor\`eme \ref{chongshu2}.

Pour d\'emontrer la proposition \ref{auxillary scheme}, on raisonne par r\'ecurrence sur la profondeur maximale $s$ des $\mathscr T_Y$, o\`u $Y\in\mathcal C(X_1\intersect X_r)$.

\textbf{\'Etape 2-1: le cas o\`u $s=1$. - }
D'abord, on va d\'emotrer le cas de $s=1$. On suppose que $\mathcal Z_*\smallsetminus\mathcal Z_0\neq\emptyset$. Dans ce cas-l\`a, on va d\'emontrer le lemme suivant.
\begin{lemm}\label{height of vertex=1}
  Soient $M\in\mathcal Z$, et $k'/k$ une extension galoisienne finie de corps telle que $\#k'\geqslant\deg(X_1)\cdots\deg(X_r)$ et $M^{\mathrm{reg}}(k')\neq\emptyset$. Pour tout $M'\in\mathcal C(M_{k'})$, il existe un cylindre $M^{0}_{k'}$ dans $\mathbb P(E_{k'})$, tel que $M'\in\mathcal C(X_{1,k'}\intersect X_{r,k'}\cdot M^0_{k'})$ et que les sch\'emas $X_{1,k'},\ldots,X_{r,k'},M^0_{k'}$ s'intersectent proprement en la composante irr\'eductible $M'$. De plus, l'in\'egalit\'e
  \begin{eqnarray*}
    & &\sum_{Y\in\mathcal C(X_1\intersect X_r)}i(Y;X_1\intersect X_r;\mathbb P(E))W_{\mathscr T_Y}(M)\\
    &\geqslant& i(M';X_{1,k'}\intersect X_{r,k'}\cdot M^0_{k'};\mathbb P(E_{k'}))
  \end{eqnarray*}
  est v\'erifi\'ee.
\end{lemm}
\begin{proof}
 Pour tout $Y\in \mathcal C(X_1\intersect X_r)$, on d\'esigne par $\widetilde{Y}$ l'\'etiquette de $Y$ dans l'arbre d'intersection consid\'er\'e comme un sch\'ema. Par la d\'efinition dans \S\ref{definition of intersection tree}, on a
 \begin{equation}\label{definition of W_T}
W_{\mathscr T_Y}(M)=i(M;Y\cdot \widetilde{Y};\mathbb P(E)).
\end{equation}
En effet, comme $s=1$, si $M$ appara\^it dans les d\'escendants de $Y$, il n'appara\^it qu'une seule fois. De plus, on a
\begin{equation}\label{s=1-comparation of mult}
i(M;Y\cdot \widetilde{Y};\mathbb P(E))\geqslant\mu_M(Y)\mu_M(\widetilde{Y})\geqslant\mu_M(Y)
\end{equation}
d'apr\`es la proposition \ref{inqmult}. Par la proposition \ref{mult of point under base change}, on a
\begin{equation}\label{invariance under base change}
\mu_M(Y)=\mu_{M'}(Y_{k'}).
\end{equation}

Comme $k$ est un corps parfait, le sch\'ema $Y_{k'}$ est r\'eduit d'apr\`es \cite[Proposition 3.2.7]{LiuQing}. Donc $\O_{Y_{k'},Y'}$ est un anneau local artinien r\'eduit, qui est un corps (cf. \cite[Proposition 8.9]{AM_commutative_algebra}, l'id\'eal maximal de $\O_{Y_{k'},Y'}$ est nul). On en d\'eduit $\ell_{\O_{Y_{k'},Y'}}(\O_{Y_{k'},Y'})=1$. D'apr\`es la proposition \ref{sumofmult}, on a
\begin{equation}\label{sum of multiplicity}
\mu_{M'}(Y_{k'})=\sum_{Y'\in\mathcal C(Y_{k'})}\mu_{M'}(Y').
\end{equation}

On obtient donc
\begin{eqnarray*}
& &\sum_{Y\in \mathcal C(X_1\intersect X_r)}i(Y;X_1\intersect X_r;\mathbb P(E))W_{\mathscr T_{Y}}(M)\\
&\geqslant&\sum_{Y\in\mathcal C(X_1\intersect X_r)}i(Y;X_1\intersect X_r;\mathbb P(E))\sum_{Y'\in\mathcal C(Y_{k'})}\mu_{M'}(Y')
\end{eqnarray*}
d'apr\`es les in\'egalit\'es \eqref{definition of W_T}, \eqref{s=1-comparation of mult}, \eqref{invariance under base change}, et \eqref{sum of multiplicity}.

En outre, on a
   \begin{eqnarray}\label{why we assume the number of k'}
   & &\sum_{Y'\in\mathcal C(X_{1,k'}\intersect X_{r,k'})}\deg (Y')\\
   &\leqslant&\sum_{Y'\in\mathcal C(X_{1,k'}\intersect X_{r,k'})}i(Y';X_{1,k'}\intersect X_{r,k'};\mathbb P(E_{k'}))\deg (Y')\nonumber\\
   &=&\sum_{Y\in\mathcal C(X_1\intersect X_r)}\sum_{Y'\in\mathcal C(Y_{k'})}i(Y';X_{1,k'}\intersect X_{r,k'};\mathbb P(E_{k'}))\deg (Y')\nonumber\\
   &=&\deg(X_{1,k'})\cdots\deg(X_{r,k'})\nonumber\\
   &=&\deg(X_1)\cdots\deg(X_r)\nonumber,
   \end{eqnarray}
   o\`u la premi\`ere \'egalit\'e provient de la proposition \ref{basechangemultinter}, la deuxi\`eme \'egalit\'e provient du th\'eor\`eme de B\'ezout (le th\'eor\`eme \ref{bezout}), et la derni\`ere \'egalit\'e provient du fait que le degr\'e d'un sous-sch\'ema ferm\'e de $\mathbb P(E)$ est invariant sous extension de corps. Donc on a
    \[\#k'\geqslant\sum_{Y'\in\mathcal C(X_{1,k'}\intersect X_{r,k'})}\deg (Y')\]
d'apr\`es l'in\'egalit\'e \eqref{why we assume the number of k'}.

     On d\'egine par $\mathcal D(M)$ le sous-ensemble de $Y\in\mathcal C(X_1\intersect X_r)$ tel que $M$ appara\^it \`a une d\'escendant de $Y$ dans l'arbre d'intersection $\mathscr T_Y$.

 La composante $M$ admet un $k'$-point r\'egulier. Comme $k$ est parfait, d'apr\`es la proposition \ref{sumofmult}, la composante $M'$ admet un $k'$-point rationnel $P$ de multiplicit\'e $1$, qui est r\'egulier car $M_{k'}$ est de dimension pure. D'apr\`es la proposition \ref{numofproper}, on obtient qu'il existe un sous-sch\'ema $k'$-lin\'eaire ferm\'e de $\mathbb P(E_{k'})$ de dimension $n-\dim (Y)=n-\dim(Y_{k'})$ qui intersecte tous les $Y'\in\bigcup\limits_{Y\in\mathcal D(M)}\mathcal C(Y_{k'})$ proprement en le point $P$ ou en les composantes qui ne contiennent pas $P$. Dans ce cas-l\`a, ce sous-sch\'ema $k'$-lin\'eaire ferm\'e de $\mathbb P(E_{k'})$ intersecte $M'$ en ce point $k'$-r\'egulier de $M'$ seulement au voisinage de ce point. D'apr\`es la proposition \ref{cylindre}, on peut trouver un cylindre $M^0_{k'}$ de dimension $n-\dim(Y)+\dim(M)=n-\dim(Y')+\dim(M')$ dont la direction est d\'efinie par ce sous-sch\'ema $k'$-lin\'eaire ferm\'e de $\mathbb P(E_{k'})$, tel qu'il intersecte tous les $Y'\in\bigcup\limits_{Y\in\mathcal D(M)}\mathcal C(Y_{k'})$ proprement en la composante $M'$ ou en les composantes irr\'eductible qui ne contiennent pas $M'$. De plus, on a
\[\mu_{M'}(Y')=i( M';Y'\cdot M^0_{k'};\mathbb P(E_{k'}))\]
pour toute composante irr\'eductible $Y'\in\mathcal C(Y_{k'})$, $Y\in\mathcal C(X_1\intersect X_r)$.

D'apr\`es le lemme \ref{irreducible component change after extension of fields} et la proposition \ref{basechangemultinter}, on a
\begin{equation}\label{intersection multiplicity of Y}
  i(Y;X_1\intersect X_r;\mathbb P(E))=i(Y';X_{1,k'}\intersect X_{r,k'};\mathbb P(E_{k'})),
\end{equation}
o\`u $Y'\in \mathcal C(Y_{k'})$. Donc on obtient
\begin{eqnarray*}
  & &\sum_{Y\in\mathcal C(X_1\intersect X_r)}i(Y;X_1\intersect X_r;\mathbb P(E))\sum_{Y'\in\mathcal C(Y_{k'})}\mu_{M'}(Y')\\
  &=&\sum_{Y\in\mathcal C(X_1\intersect X_r)}\sum_{Y'\in\mathcal C(Y_{k'})}i(Y';X_{1,k'}\intersect X_{r,k'};\mathbb P(E_{k'}))i(M';Y'\cdot M^0_{k'};\mathbb P(E_{k'}))
\end{eqnarray*}
d'apr\`es l'\'egalit\'e \eqref{intersection multiplicity of Y}. Par la d\'efinition de $\mathcal Z_*$ (la d\'efinition \ref{Zs}), les composantes irr\'eductibles dans $\mathcal C(X_1\intersect X_r)\smallsetminus\mathcal D(M)$ ne contiennent pas $M$. Donc le cylindre $M^0_{k'}$ n'intersecte pas les composantes irr\'eductibles de l'intersection $X_{1,k'}\intersect X_{r,k'}$ dans $\mathcal C(X_{1,k'}\intersect X_{r,k'})\smallsetminus\{N\in\mathcal C(Y_{k'})|\;Y\in\mathcal D(M)\}$ en la comsposante $M'$. Alors d'apr\`es l'associativit\'e d'intersection propre (l'\'enonc\'e (ii) de la proposition \ref{comm-ass}), on a
\begin{eqnarray*}
  & &\sum_{Y\in\mathcal C(X_1\intersect X_r)}\sum_{Y'\in\mathcal C(Y_{k'})}i(Y';X_{1,k'}\intersect X_{r,k'};\mathbb P(E_{k'}))i(M';Y'\cdot M^0_{k'};\mathbb P(E_{k'}))\\
  &=&i(M';X_{1,k'}\intersect X_{r,k'}\cdot M^0_{k'};\mathbb P(E_{k'})).
\end{eqnarray*}
C'est la fin de la preuve du lemme \ref{height of vertex=1}, qui d\'emontre la proposition \ref{auxillary scheme} pour le cas de $s=1$.
\end{proof}
\textbf{\'Etape 2-2: du cas o\`u la profondeur maximale est $s-1$ au cas o\`u la profondeur maximale est $s$. - } Afin de d\'emontrer la proposition \ref{auxillary scheme}, on raisonne par r\'ecurrence sur la profondeur maximale des arbres d'intersection $\mathscr T_Y$, o\`u $Y\in\mathcal C(X_1\intersect X_r)$. On rappelle que l'on prend une extension galoisienne finie de corps $k'/k$ telle que
\[\#k'\geqslant\delta^{\sum\limits_{i=1}^r\dim(X_i)-(r-1)(n-1)}\]
 et que $M^{\mathrm{reg}}(k')\neq\emptyset$, o\`u $n=\rg_k(E)-1$.
\begin{proof}[D\'emonstration de la proposition \ref{auxillary scheme}]
Dans cette d\'emonstration, on maintient toutes les notations dans la d\'emonstration du lemme \ref{height of vertex=1}. On raisonne par r\'ecurrence sur la profondeur maximale $s$ des arbres d'intersection $\mathscr T_Y$, o\`u $Y\in\mathcal C(X_1\intersect X_r)$. Le cas de $s=1$ est d\'emontr\'e dans le lemme \ref{height of vertex=1}.

 Dans un arbre d'intersection, un fils d'un sommet est de codimension plus grande ou \'egale \`a $1$ dans ce sommet, donc on obtient que la valeur maximal de $s$ est $\dim(Y)$. Pour tout $Y\in\mathcal C(X_1\intersect X_r)$, on a $\dim(Y)=\sum\limits_{i=1}^r\dim(X_i)-(r-1)n$.

Maintenant on suppose que l'assertion est d\'emontr\'ee pour le cas o\`u la profondeur maximale des $\mathscr T_Y$ est $s-1$ pour tous les $Y\in\mathcal C(X_1\intersect X_r)$. Dans la suite, on d\'emontre le cas o\`u la profondeur maximale des $\{\mathscr T_Y\}_{Y\in\mathcal C(X_1\intersect X_r)}$ est $s$. Pour tout $N\in \mathcal C_*$, on d\'esigne par $\widetilde{N}$ l'\'etiquette de $N$ dans l'arbre d'intersection $\mathscr T_Y$.

Pour tous $N\in\mathcal C_*$, par la condition $\#k'\geqslant\delta^{\sum\limits_{i=1}^r\dim(X_i)-(r-1)(n-1)}$, on obtient l'in\'egalit\'e
\[\#k'\geqslant\deg(N)\deg(\widetilde{N})\]
d'apr\`es le th\'eor\`eme de B\'ezout (le th\'eor\`eme \ref{bezout}). Donc par le lemme \ref{height of vertex=1}, on peut utiliser l'hypoth\`ese de r\'ecurrence aux sous-arbres d'intersection des $\{\mathscr T_Y\}_{Y\in\mathcal C(X_1\intersect X_r)}$, dont les racines sont les sommets dans $\mathcal C_1$. D'apr\`es l'hypoth\`ese de r\'ecurrence et la d\'efinition \ref{Zs}, pour tout $Y\in\mathcal D(M)$, on peut trouver un cylindre $Z_Y$ dans $\mathbb P(E_{k'})$ de dimension $n-\dim(Y)+\dim(M)$, tel que $Y_{k'}$, $\widetilde{Y}_{k'}$ et $Z_Y$ s'intersectent proprement en $M'$, et
\begin{equation}\label{induction hypothesis inequality}
W_{\mathscr T_Y}(M)=\sum_{Y'\in\mathcal C(Y\cdot\widetilde{Y})}i(Y';Y\cdot\widetilde{Y};\mathbb P(E))W_{\mathscr T_{Y'}}(M)\geqslant i(M';Y_{k'}\cdot \widetilde{Y}_{k'}\cdot Z_Y;\mathbb P(E_{k'})),
\end{equation}
o\`u $\mathscr T_{Y'}$ est le sous-arbre d'intersection dont la racine est $Y'\in\mathcal C(Y\cdot\widetilde{Y})$.

Dans la suite, on estimera la multiplicit\'e d'intersection $ i(M';Y_{k'}\cdot \widetilde{Y}_{k'}\cdot Z_Y;\mathbb P(E_{k'}))$. Comme $k$ est un corps parfait, le sch\'ema $Y_{k'}$ est r\'eduit. D'o\`u l'on a
\begin{eqnarray}\label{induction hypothesis}
W_{\mathscr T_Y}(M)&\geqslant&i(M';Y_{k'}\cdot \widetilde{Y}_{k'}\cdot Z_Y;\mathbb P(E_{k'}))\nonumber\\
  &\geqslant&\mu_{M'}(Y_{k'})\mu_{M'}(\widetilde{Y}_{k'})\mu_{M'}(Z_Y)\nonumber\;\text{ (la proposition \ref{inqmult})}\\
  &\geqslant&\mu_{M'}(Y_{k'})\nonumber\\
  &=&\sum_{Y'\in\mathcal C(Y_{k'})}\mu_{M'}(Y')\;\text{ (la proposition \ref{sumofmult})}
\end{eqnarray}
d'apr\`es l'in\'egalit\'e \eqref{induction hypothesis inequality}, car l'anneau $\O_{Y_{k'},Y'}$ est un anneau local artinien r\'eduit, qui est un corps (cf. \cite[Proposition 8.9]{AM_commutative_algebra}, l'id\'eal maximal de $\O_{Y_{k'},Y'}$ est nul). Donc on obtient
\begin{eqnarray*}
& &\sum_{Y\in \mathcal C(X_1\intersect X_r)}i(Y;X_1\intersect X_r;\mathbb P(E))W_{\mathscr T_{Y}}(M)\\
&\geqslant&\sum_{Y\in\mathcal C(X_1\intersect X_r)}i(Y;X_1\intersect X_r;\mathbb P(E))\sum_{Y'\in\mathcal C(Y_{k'})}\mu_{M'}(Y')\\
&=&\sum_{Y\in\mathcal C(X_1\intersect X_r)}\sum_{Y'\in\mathcal C(Y_{k'})}i(Y';X_{1,k'}\intersect X_{r,k'};\mathbb P(E_{k'}))\mu_{M'}(Y')\\
&=&\sum_{Y'\in\mathcal C(X_{1,k'}\intersect X_{r,k'})}i(Y';X_{1,k'}\intersect X_{r,k'};\mathbb P(E_{k'}))\mu_{M'}(Y')
\end{eqnarray*}
d'apr\`es l'in\'egalit\'e \eqref{induction hypothesis}, le lemme \ref{irreducible component change after extension of fields} et la proposition \ref{basechangemultinter}.

En outre, d'apr\`es l'in\'egalit\'e \eqref{why we assume the number of k'}, on a
    \[\sum_{Y'\in\mathcal C(X_{1,k'}\intersect X_{r,k'})}\deg (Y')\leqslant\#k'.\]

    La composante $M$ admet un $k'$-point r\'egulier. Comme $k$ est parfait, d'apr\`es la proposition \ref{sumofmult}, la composante $M'$ admet un $k'$-point rationnel $P$ de multiplicit\'e $1$, qui est r\'egulier car $M_{k'}$ est de dimension pure. D'apr\`es la proposition \ref{numofproper}, on obtient qu'il existe un sous-sch\'ema $k'$-lin\'eaire ferm\'e de $\mathbb P(E_{k'})$ de dimension $n-\dim (Y)=n-\dim(Y_{k'})$ qui intersecte tout $Y'\in\bigcup\limits_{Y\in\mathcal D(M)}\mathcal C(Y_{k'})$ proprement en le point $P$ ou en les composante qui ne contiennet pas $P$. Dans ce cas-l\`a, ce sous-sch\'ema $k'$-lin\'eaire ferm\'e de $\mathbb P(E_{k'})$ intersecte $M'$ seulement au voisinage de $P$. D'apr\`es la proposition \ref{cylindre}, on peut trouver un cylindre $M^0_{k'}$ de dimension $n-\dim(Y)+\dim(M)=n-\dim(Y')+\dim(M')$ dont la direction est d\'efinie par ce sous-sch\'ema $k'$-lin\'eaire ferm\'e de $\mathbb P(E_{k'})$, tel qu'il intersecte tous les $Y'\in\bigcup\limits_{Y\in\mathcal D(M)}\mathcal C(Y_{k'})$ proprement en la composante $M'$ ou en les composantes qui ne contiennent pas $M'$. De plus, on a
\[\mu_{M'}(Y')=i( M';Y'\cdot M^0_{k'};\mathbb P(E_{k'}))\]
pour toute composante irr\'eductible $Y'\in\mathcal C(Y_{k'})$. Donc on obtient
\begin{eqnarray*}
  & &\sum_{Y'\in\mathcal C(X_{1,k'}\intersect X_{r,k'})}i(Y';X_{1,k'}\intersect X_{r,k'};\mathbb P(E_{k'}))\mu_{M'}(Y')\\
  &=&\sum_{Y'\in\mathcal C(X_{1,k'}\intersect X_{r,k'})}i(Y';X_{1,k'}\intersect X_{r,k'};\mathbb P(E_{k'}))i(M';Y'\cdot M^0_{k'};\mathbb P(E_{k'}))
\end{eqnarray*}
d'\'apr\`es l'\'egalit\'e \eqref{intersection multiplicity of Y}. Par la d\'efinition de $\mathcal Z_*$ (la d\'efinition \ref{Zs}), les composantes irr\'eductibles dans $\mathcal C(X_1\intersect X_r)\smallsetminus\mathcal D(M)$ ne contiennent pas $M$. Donc le cylindre $M^0_{k'}$ n'intersecte pas les composantes irr\'eductibles de l'intersection $X_{1,k'}\intersect X_{r,k'}$ dans $\mathcal C(X_{1,k'}\intersect X_{r,k'})\smallsetminus\{N\in\mathcal C(Y_{k'})|\;Y\in\mathcal D(M)\}$ en la composante $M'$. Alors d'apr\`es l'associativit\'e d'intersection propre (l'\'enonc\'e (ii) de la proposition \ref{comm-ass}), on a
\begin{eqnarray*}
  & &\sum_{Y'\in\mathcal C(X_{1,k'}\intersect X_{r,k'})}i(Y';X_{1,k'}\intersect X_{r,k'};\mathbb P(E_{k'}))i(M';Y'\cdot M^0_{k'};\mathbb P(E_{k'}))\\
  &=&i(M';X_{1,k'}\intersect X_{r,k'}\cdot M^0_{k'};\mathbb P(E_{k'})).
\end{eqnarray*} Alors on d\'emontre que la proposition \ref{auxillary scheme} est vraie pour le cas o\`u la profondeur maximale est $s$. C'est la fin de la d\'emonstration de la proposition \ref{auxillary scheme}.
\end{proof}
\subsection*{Une cons\'equence du th\'eor\`eme \ref{chongshu2}}
On a d\'emontr\'e le th\'eor\`eme \ref{chongshu2}. Dans la suite, on va d\'eduire une cons\'equence du th\'eor\`eme \ref{chongshu2}, qui donne une majoration globale des multiplicit\'es des sommets dans $\mathcal Z_*$. Cette majoration sera utilis\'ee dans la d\'emonstration du th\'eor\`eme principal (le th\'eor\`eme \ref{main result}).
\begin{defi}\label{Z'_s-C'_s}
  Soit $s$ un entier positif. On d\'esigne par $\mathcal C'_s$ (resp. $\mathcal Z'_s$, $\mathcal C'_*$, et $\mathcal Z'_*$) l'ensemble des \'etiquettes de $\mathcal C_s$ (resp. $\mathcal Z_s$, $\mathcal C_*$, et $\mathcal Z_*$), voir les d\'efinitions \ref{Cs} et \ref{Zs} pour les d\'efinition de $\mathcal C_s$, $\mathcal Z_s$, $\mathcal C_*$, et $\mathcal Z_*$.
\end{defi}
Avec toutes les notations ci-dessus, si tous les \'etiquettes non-vides dans $\mathscr T_Y$ ont la m\^eme dimension, pour les sommets dans $\mathcal Z_*$, on a le corollaire suivant qui est une description globale de ses multiplicit\'es dans $X_1,\ldots,X_r$.
\begin{prop}\label{globale}
  Avec les notations et conditions dans le th\'eor\`eme \ref{chongshu2}, on suppose que tous les \'el\'ements non-vides dans $\mathcal C'_*$ ont la m\^eme dimension. Alors on a
  \begin{equation*}
    \sum_{Z\in \mathcal Z_s} \left(\prod_{i=1}^r\mu_Z(X_i)\right)\deg(Z)\leqslant\prod_{i=1}^r\deg(X_i)\prod_{j=0}^{s-1}\max_{\widetilde{Y}\in \mathcal C'_j}\{\deg(\widetilde{Y})\}.
  \end{equation*}
En particulier, si $s=0$, on d\'efinit $\prod\limits_{j=0}^{s-1}\max\limits_{Y'\in \mathcal C'_j}\{\deg(Y')\}=1$.
\end{prop}
\begin{proof}
Comme les $Y\in\mathcal C(X_1\intersect X_r)$ sont de la m\^eme dimension et ses \'etiquettes sont de la m\^eme dimension, les sommets de profondeur $1$ dans les $\mathscr T_Y$ sont de la m\^eme dimension puisque $Y$ intersecte son \'etiquette proprement pour tout $Y\in\mathcal C(X_1\intersect X_r)$. Par le m\^eme argument ci-dessus, pour un entier positif $s$ fix\'e, les sommets dans $\mathcal C_s$ sont de la m\^eme dimension.

Afin de d\'emontrer cette proposition, d'abord on raisonne par r\'ecurrence sur la profondeur $s$ pour montrer l'in\'egalit\'e
\begin{eqnarray*}
& &\sum_{Z\in\mathcal C_s}\left(\sum_{Y\in\mathcal C(X_1\intersect X_r)}i(Y;X_1\intersect X_r;\mathbb P(E))W_{\mathscr T_Y}(Z)\right)\deg(Z)\\
&\leqslant&\prod_{i=1}^r\deg(X_i)\prod_{j=0}^{s-1}\max_{\widetilde{Y}\in \mathcal C'_j}\{\deg(\widetilde{Y})\}.
\end{eqnarray*}

D'apr\`es le th\'eor\`eme de B\'ezout (le th\'eor\`eme \ref{bezout}), on a
\begin{equation*}
  \sum_{Y\in\mathcal C(X_1\intersect X_r)}i(Y_i;X_1\intersect X_r;\mathbb P(E))\deg(Y)=\deg(X_1)\deg(X_2)\cdots\deg(X_r),
\end{equation*}
ce qui montre le cas de $s=0$.

On suppose que le cas de profondeur $s-1$ est d\'emontr\'e. Pour le cas de profondeur $s$, on a
\begin{eqnarray*}
  & &\prod_{i=1}^r\deg(X_i)\prod_{j=0}^{s-1}\max_{\widetilde{Y}\in \mathcal C'_j}\{\deg(\widetilde{Y})\}\\
  &\geqslant&\max_{\widetilde{Y}\in \mathcal C'_{s-1}}\{\deg(\widetilde{Y})\}\sum_{Z\in\mathcal C_{s-1}}\left(\sum_{Y\in\mathcal C(X_1\intersect X_r)}i(Y;X_1\intersect X_r;\mathbb P(E))W_{\mathscr T_Y}(Z)\right)\deg(Z)
  \end{eqnarray*}
 d'apr\`es le th\'eor\`eme de B\'ezout (le th\'eor\`eme \ref{bezout}). Pour tout sommet $Z$, on d\'esigne par $\widetilde{Z}$ l'\'etiquette de $Z$. Donc on obtient
  \begin{eqnarray*}
  & &\max_{\widetilde{Y}\in \mathcal C'_{s-1}}\{\deg(\widetilde{Y})\}\sum_{Z\in\mathcal C_{s-1}}\left(\sum_{Y\in\mathcal C(X_1\intersect X_r)}i(Y;X_1\intersect X_r;\mathbb P(E))W_{\mathscr T_Y}(Z)\right)\deg(Z)\\
  &\geqslant&\sum_{Z\in\mathcal C_{s-1}}\left(\sum_{Y\in\mathcal C(X_1\intersect X_r)}i(Y;X_1\intersect X_r;\mathbb P(E))W_{\mathscr T_Y}(Z)\right)\deg(Z)\deg(\widetilde{Z})\\
  &=&\sum_{Z\in\mathcal C_{s-1}}\left(\sum_{Y\in\mathcal C(X_1\intersect X_r)}i(Y;X_1\intersect X_r;\mathbb P(E))W_{\mathscr T_Y}(Z)\right)\cdot\\
  & &\sum_{Z'\in\mathcal C (Z\cdot\widetilde{Z})}i(Z';Z\cdot\widetilde{Z};\mathbb P(E))\deg(Z')\\
  &=&\sum_{Z'\in\mathcal C_{s}}\left(\sum_{Y\in\mathcal C(X_1\intersect X_r)}i(Y;X_1\intersect X_r;\mathbb P(E))W_{\mathscr T_Y}(Z')\right)\deg(Z'),
\end{eqnarray*}
ce qui montre le cas de profondeur $s$.

Dans la suite, il faut d\'emontrer l'in\'egalit\'e
\begin{eqnarray*}
  & &\sum_{Z\in \mathcal Z_s}\left( \prod_{i=1}^r\mu_Z(X_i)\right)\deg(Z)\\
  &\leqslant&\sum_{Z\in\mathcal C_s}\left(\sum_{Y\in\mathcal C(X_1\intersect X_r)}i(Y;X_1\intersect X_r;\mathbb P(E))W_{\mathscr T_Y}(Z)\right)\deg(Z).
\end{eqnarray*}
Pour un $Z\in \mathcal Z_s$ fix\'e, d'apr\`es le th\'eor\`eme \ref{chongshu2}, on obtient
\begin{equation*}
  \sum_{Y\in\mathcal C(X_1\intersect X_r)}i(Y;X_1\intersect X_r;\mathbb P(E))W_{\mathscr T_Y}(Z)\geqslant \mu_{Z}(X_1)\cdots\mu_Z(X_r).
  \end{equation*}
Par la d\'efinition \ref{Zs}, l'ensemble $\mathcal Z_s$ est un sous-ensemble de $\mathcal C_s$ pour tout $s\geqslant0$. Donc on obtient le r\'esultat.
\end{proof}

\section{Estimation de multiplicit\'es dans une hypersurface}
Le r\'esultat suivant est une majoration d'un comptage de multiplicit\'es dans une hypersurface projective r\'eduite sur le corps fini $\f_q$. Cette majoration peut \^etre consid\'er\'ee comme une description de la complexit\'e du lieu singulier de cette hypersurface projective r\'eduite.
\begin{theo}\label{main result}
  Soient $X\hookrightarrow \mathbb{P}^n_{\f_q}$ une hypersurface projective r\'eduite de degr\'e $\delta$, o\`u $\dim (X^{\mathrm{sing}})=s$. Alors on a:
  \begin{eqnarray*}
    \sum\limits_{\xi\in X(\f_q)}\mu_\xi(X)(\mu_\xi(X)-1)^{n-s-1}&\leqslant&\delta(\delta-1)^{n-s-1}(q^{s}+q^{s-1}+\cdots+1)+\\
    & &\delta(\delta-1)^{n-s}(q^{s-1}+q^{s-2}+\cdots+1)+\cdots\\
    & &+\delta(\delta-1)^{n-1},
  \end{eqnarray*}
o\`u $\mu_\xi(X)$ est la multiplicit\'e de $\xi$ dans $X$ (voir \S \ref{mult of a sub-scheme} pour la d\'efinition).
\end{theo}
Avant de la d\'emontration du th\'eor\`eme \ref{main result}, on a besoin d'introduire certaines propri\'et\'es sp\'eciales autour de la multiplicit\'e dans une section par hypersurface, et introduire une m\'ethode de construire les arbres d'intersection utiles pour ce probl\`eme de comptage de multiplicit\'es.
\subsection{Multiplicit\'e dans une section par hypersurface}
Soient $k$ un corps, et $f\in k[T_0,\ldots,T_n]$ un polyn\^ome homog\`ene non-nul de degr\'e $\delta$. On dit que le sch\'ema
\[X=\proj\left(k[T_0,\ldots,T_n]/(f)\right)\]
 est une \textit{hypersurface projective} (ou \textit{hypersurface} pour simplifier) de $\mathbb P^n_k$ d\'efinie par le polyn\^ome $f$. On peut d\'emontrer que $X$ est un sous-sch\'ema ferm\'e de degr\'e $\delta$ de $\mathbb P^n_k$ (cf. \cite[Proposition 7.6, Chap. I]{GTM52}).

On va introduire certaines propri\'et\'es sp\'eciales autour de la multiplicit\'e d'un point dans une hypersurface projective.
\begin{prop}[\cite{Kollar2007}, Example 2.70 (2)]\label{local hilbert of hypersurface}
Soient $X$ une hypersurface de $\mathbb P^n_k$ d\'efinie par un polyn\^ome homog\`ene $f$ non-nul, $\xi\in X(\overline k)$, et $\sm_\xi$ l'id\'eal maximal d'anneau local $\O_{\mathbb{P}^n_k,\xi}$. Soit $H_\xi(s)$ la fonction de Hilbert-Samuel locale de $X$ en le point $\xi$ (voir \S \ref{multiplicity of local ring} pour la d\'efinition). Si l'image de $f$ dans $\O_{\mathbb P^n_k,\xi}$ appara\^it \`a l'ensemble $\sm_\xi^r\smallsetminus \sm_\xi^{r+1}$. Alors on a
\begin{equation*}
  H_\xi(s)={n+s-1\choose s}-{n+s-r-1\choose s-r}.
\end{equation*}
En particulier, on a $\mu_\xi(X)=r$.
\end{prop}

Soit $I=(i_0,\ldots,i_n)\in\mathbb N^{n+1}$ un indice, on d\'efinit $|I|=i_0+\cdots+i_n$. Soit $g(T_0,\ldots,T_n)$ un polyn\^ome homog\`ene non-nul de degr\'e $\delta$, alors on peut d\'evelopper le polyn\^ome $g(T_0+S_0,T_1+S_1,\ldots,T_n+S_n)\in k[T_0,T_1,\ldots,T_n,S_0,S_1,\ldots,S_n]$ comme
\begin{eqnarray*}
  & &g(T_0+S_0,\ldots,T_n+S_n)\\
  &=&g(T_0,\ldots,T_n)+\sum_{\alpha=1}^\delta\sum_{\begin{subarray}{c}I=(i_0,\ldots,i_n)\in\mathbb N^{n+1}\\|I|=\alpha\end{subarray}}g^{I}(T_0,\ldots,T_n)S_0^{i_0}\cdots S_n^{i_n},
\end{eqnarray*}
o\`u $g^{I}(T_0,\ldots,T_n)$ est un polyn\^ome homog\`ene de degr\'e $\delta-|I|$ ou nul. On d\'esigne par $\mathcal D^\alpha(g)$ l'ensemble des polyn\^omes $g^I(T_0,\ldots,T_n)$ d\'efinis ci-dessus, o\`u $|I|=\alpha\geqslant1$.

 Pour un entier $1\leqslant\alpha\leqslant\delta$, on d\'efinit $\mathcal T^{\alpha}(g)$ comme l'espace $k$-vectoriel engendr\'e par les \'el\'ements dans $\mathcal D^\alpha(g)$. Pour tout $g\in\mathcal T^{\alpha}(g)$ non-nul, $g$ d\'efinit une hypersurface projective de degr\'e $\delta-\alpha$ de $\mathbb P^n_k$.

 De plus, on d\'efinit $\mathcal D^0(g)=\{g\}$ et $\mathcal T^0(g)=k\cdot g$.
\begin{rema}
Les \'el\'ements dans $\mathcal D^1(g)$ sont les
\[\frac{\partial g}{\partial T_0},\frac{\partial g}{\partial T_1},\ldots,\frac{\partial g}{\partial T_n},\]
qui sont des polyn\^omes homog\`enes de degr\'e $\delta-1$ ou nuls. Si $\car(k)=0$ ou $\car(k)>\delta$, les \'el\'ements dans $\mathcal D^\alpha(g)$ ont la forme de
  \[\frac{1}{i_0!\cdots i_n!}\cdot\frac{\partial^{i_0+\cdots+i_n}g(T_0,\ldots,T_n)}{\partial T_0^{i_0}\cdots\partial T_n^{i_n}},\]
   o\`u $(i_0,\ldots,i_n)\in\mathbb N^{n+1}$ est une indice avec $i_0+\cdots+i_n=\alpha$. De plus, l'espace $k$-vectoriel $\mathcal T^{\alpha}(g)$ est l'espace des d\'eriv\'ees directionnelles de l'ordre $\alpha$ de $g(T_0,\ldots,T_n)$.
\end{rema}
Avec toutes les notations ci-dessus, on a une cons\'equence directe de la proposition \ref{local hilbert of hypersurface} suivant:
\begin{coro}\label{point}
  Soient $X\hookrightarrow\mathbb{P}^n_k$ l'hypersurface projective d\'efinie par un polyn\^ome homog\`ene $f\neq0$ de degr\'e $\delta$, $\xi\in X(\overline{k})$, et $\alpha$ un entier tel que $0\leqslant\alpha\leqslant \mu_\xi(X)-1$. Alors pour tout $g\in \mathcal T^\alpha(f)$ non-nul, le point $\xi$ est contenu dans l'hypersurface d\'efinie par $g$. Il existe un $g'\in \mathcal T^{\mu_\xi(X)}(f)$ non-nul, tel que $\xi$ n'est pas contenu dans l'hypersurface d\'efinie par $g'$.
\end{coro}
\begin{proof}
	Soit $\xi=[a_0:\cdots:a_n]$. D'apr\`es la proposition \ref{local hilbert of hypersurface}, l'image de $f$ dans l'anneau local $\O_{\mathbb P^n_k,\xi}$ est dans l'ensemble $\sm_{\xi}^{\mu_\xi(X)}\smallsetminus\sm_\xi^{\mu_\xi(X)+1}$, qui signifie que cette image est dans $\sm_{\xi}^{\mu_\xi(X)}$ mais n'est pas dans $\sm_{\xi}^{\mu_\xi(X)+1}$. L'image \'etant dans $\sm_{\xi}^{\mu_\xi(X)}$ signifie que pour tout polyn\^ome $f^{I}(T_0,\ldots,T_n)$ d\'efini ci-dessus avec $0\leqslant|I|\leqslant\mu_\xi(X)-1$, on a $f^{I}(a_0,\ldots,a_n)=0$. L'image n'\'etant pas dans $\xi\not\in\sm_{\xi}^{\mu_\xi(X)+1}$ signifie qu'il existe un polyn\^ome $f^{I}(T_0,\ldots,T_n)$ avec $|I|=\mu_\xi(X)$ tel que $f^{I}(a_0,\ldots,a_n)\neq0$. Alors on a l'assertion.
\end{proof}
Une cons\'equence directe du corollaire \ref{point} est ci-dessous.
\begin{coro}\label{mult}
  Soient $X\hookrightarrow\mathbb{P}^n_k$ l'hypersurface projective d\'efinie par un polyn\^ome homog\`ene $f$ de degr\'e $\delta$, et $\eta\in X$ un point sch\'ematique. Pour un entier $0\leqslant\alpha\leqslant\delta$, soit $X'$ l'hypersurface de $\mathbb P^n_k$ d\'efinie par un \'el\'ement non-nul $g\in\mathcal T^{\alpha}(f)$, o\`u $\alpha<\mu_\eta(X)$. Alors la multiplicit\'e $\mu_\eta(X')$ est au moins $\mu_\eta(X)-\alpha$. De plus, il existe au moins un \'el\'ement dans $\mathcal T^{\alpha}(f)$ qui d\'efinit une hypersurface $X''$ de $\mathbb P^n_k$, telle que la multiplicit\'e $\mu_\eta(X'')$ soit \'egale \`a $\mu_\xi(X)-\alpha$.
\end{coro}
\begin{proof}
  Soient $Z=\{\eta\}$, et $\xi\in Z^{\mathrm{reg}}(\overline k)$. D'apr\`es le corollaire \ref{sub}, on a $\mu_\xi(X)=\mu_\eta(X)$. Comme $Z^{\mathrm{reg}}(\overline k)$ est dense dans $Z$ (cf. \cite[Corollary 8.16, Chap. II]{GTM52}), on a l'assertion.
\end{proof}

\begin{rema}
D'apr\`es le corollaire \ref{point}, si $X\hookrightarrow\mathbb P^n_k$ est une hypersurface d\'efinie par un polyn\^ome homog\`ene non-nul de degr\'e $\delta$, la multiplicit\'e du point ferm\'e dans $X$ est au plus $\delta$.
\end{rema}
\begin{defi}\label{derivative hypersurface}
On dit que l'hypersurface projective d\'efinie par $g\in \mathcal T^\alpha(f)$ est une \textit{hypersurface d\'eriv\'ee} d'ordre $\alpha$ de l'hypersurface d\'efinie par $f$.
\end{defi}
\subsection{Construction des arbres d'intersection \`a partir d'une hypersurface}\label{construction of intersection trees}
Afin de rechercher le probl\`eme du comptage des multiplicit\'es dans une hypersurface, il faut construire quelques arbres d'intersection \`a partir de cette hypersurface. On peut \'etudier la multiplicit\'e d'un point rationnel par la majoration des poids des sommets dans les arbres d'intersection construits. Dans cette partie, soient $k$ un corps, $X$ un $k$-sch\'ema, et $k'/k$ une extension de corps, on d\'esigne par $X_{k'}$ le $k'$-sch\'ema $X\times_{\spec k}\spec k'$ pour simplifier.

D'abord, on introduit le lemme suivant, qui sera utilis\'e dans la construction des racines de ces arbres d'intersection.
\begin{lemm}\label{number of intersection}
  Soient $k$ un corps, et $g\in k[T_0,\ldots,T_n]$ un polyn\^ome homog\`ene non-nul. On d\'esigne par $V(g)$ l'hypersurface projective de $\mathbb P^n_k$ d\'efinie par $g$. Soit $f\neq0$ un polyn\^ome homog\`ene de degr\'e $\delta$. Si la dimension du lieu singulier de $V(f)$ est $s$, o\`u $0\leqslant s\leqslant n-2$. Alors il existe une extension finie $k'/k$ et une famille de $g_1,\ldots,g_{n-s-1}\in \mathcal T^1(f)\otimes_kk'$, telle que
  \[\dim\left(V(f)_{k'}\cap V(g_1)\cap \cdots \cap V(g_{n-s-1})\right)=s.\]
  Autrement dit, le sch\'ema $V(f)_{k'}\cap V(g_1)\cap \cdots \cap V(g_{n-s-1})$ est une intersection compl\`ete.
\end{lemm}
\begin{proof}
  Comme $V(f)$ prend des points singuliers, le degr\'e de $f$ est plus grand ou \'egal \`a $2$. D'abord, on suppose que $k'$ est une cl\^oture alg\'ebrique du corps $k$, alors le cardinal de $k'$ est infini. Si on d\'emontre l'assertion pour tel corps $k'$, il existe une extension finie du corps $k$ qui satisfait le besoin aussi. Dans le reste de la d\'emonstration, tous les sch\'emas que l'on consid\`ere soient d\'efinis sur cette cl\^oture alg\'ebrique du corps $k$.

  D'apr\`es le crit\`ere jacobien (cf. \cite[Theorem 4.2.19]{LiuQing}), on a
  \[\dim\left(V(f)\cap\bigcap_{g\in \mathcal T^1(f)}V(g)\right)=\dim\left(V(f)_{k'}\cap\bigcap_{g\in \mathcal T^1(f)\otimes_kk'}V(g)\right)=s.\]
On d\'esigne par $V_t$ le sch\'ema \begin{equation*}
  V(f)_{k'}\cap V(g_1)\cap\cdots\cap V(g_t)
\end{equation*}
pour simplifier. Pout tout $t\in\{0,1,\ldots,n-s-1\}$, on d\'emontrera qu'il existe $g_1,\ldots,g_t\in\mathcal T^1(f)\otimes_kk'$ (si $t=0$, on d\'efinit que l'ensemble des $\{g_1,\ldots,g_t\}$ est vide), tels que $V_t$ soit une intersection compl\`ete. Si on a l'assertion ci-dessus, on montre le r\'esultat original.

On raisonne par r\'ecurrence sur l'entier $t$ d\'efini ci-dessus, o\`u $0\leqslant t\leqslant n-s-1$. Comme $V_0=V(f)_{k'}$ est une hypersurface qui est une intersection compl\`ete, le cas de $t=0$ est d\'emontr\'e par d\'efinition directement.

 Si on a d\'ej\`a trouv\'e les $g_1,\ldots,g_t\in \mathcal T^1(f)\otimes_{k}k'$, tels que $V_t$ soit une intersection compl\`ete, o\`u $0\leqslant t\leqslant n-s-2$. Alors pour tout $U\in\mathcal C(V_t)$, on a $\dim(U)=n-t-1$.

Si pour tout $h\in\mathcal T^1(f)\otimes_k{k'}$, il toujours existe une $U\in\mathcal C(V_t)$, telle que $U\subseteq V(h)$. Alors on obtient
\[U\subsetneq V(f)_{k'}\cap \left(\bigcap_{g\in \mathcal T^1(f)\otimes_kk'}V(g)\right),\]
qui contredit avec ce que $\dim (V(f)_{k'}^{\mathrm{sing}})=s<n-t-1=\dim(U)$.

Alors pour tout $U\in\mathcal C(V_t)$, on peut trouver un $g_U\in\mathcal T^1(f)\otimes_kk'$, tel que
\[U\nsubseteq V(g_U).\]
 On d\'efinit
\[L(U)=\{h\in\mathcal T^1(f)\otimes_kk'|U\subseteq V(h)\}.\]
Alors dans ce cas-l\`a, pour tout $U\in\mathcal C(V_t)$, $L(U)$ est un sous-espace $k'$-vectoriel propre de $\mathcal T^1(f)\otimes_kk'$. Comme le cardinal de $k'$ est infini et le cardinal de $\mathcal C(V_t)$ est fini, il existe un vecteur $h\in \mathcal T^1(f)\otimes_kk'$, tel que
\[h\not\in\bigcup_{U\in \mathcal C(V_t)}L(U).\]
Alors pour tout $U\in\mathcal C(V_t)$, on a $U\nsubseteq V(h)$. Donc $V(h)\cap V_t$ est une intersection compl\`ete.

Donc pour tout $0\leqslant t\leqslant n-s-1$, on peut trouver des $g_1,\ldots,g_{t}$, qui satisfont le besoin. C'est la fin de la d\'emonstration.
\end{proof}

 Soit $X\hookrightarrow \mathbb{P}^n_{\f_q}$ l'hypersurface projective d\'efinie par le polyn\^ome homog\`ene $f$ non-nul de degr\'e $\delta$, dont la dimension du lieu singulier est plus grande ou \'egale \`a z\'ero. Soit $\f_{q^m}/\f_q$ une extension finie telle que l'on peut trouver une suite de $g_1,\ldots,g_{n-s-1}\in\mathcal T^1(f)\otimes_{\f_q}\f_{q^m}$ non-nuls qui satisfont que $X_{\f_{q^m}},V(g_1),\ldots,V(g_{n-s-1})$ soit une intersection compl\`ete. L'extension $\f_{q^m}/\f_q$ est galoisienne, car $\gal(\f_{q^m}/\f_q)=(\mathbb Z/m\mathbb Z,+)$. D'apr\`es le lemme \ref{number of intersection}, les $g_1,\ldots,g_{n-s-1}\in\mathcal T^1(f)\otimes_{\f_q}\f_{q^m}$ existent lorsque l'entier $m$ est assez positif. Soient $\xi\in X(\f_q)$, et $\xi'=\xi\times_{\spec k}\spec k'$. Alors on a $\mu_\xi(X)=\mu_{\xi'}(X_{\f_{q^m}})$ d'apr\`es la proposition \ref{mult of point under base change}.

 On d\'esigne par $X_{i,\f_{q^m}}$ l'hypersurface $V(g_i)$ d\'efinie par $g_i$ sur $\f_{q^m}$, o\`u $i=1,\ldots,n-s-1$. Par le crit\`ere jacobien (cf. \cite[Theorem 4.2.19]{LiuQing}), on obtient $X^{\mathrm{sing}}_{\f_{q^m}}\subseteq X_{\f_{q^m}}\cap X_{1,\f_{q^m}}\cap\cdots\cap X_{n-s-1,\f_{q^m}}$.

Pour tout sous-sch\'ema int\`egre $Y$ de $X_{\f_{q^m}}$, on d\'esigne par $Y^{(a)}$ le lieu des points dans $Y$ dont les multiplicit\'es sont \'egales \`a $\mu_Y(X_{\f_{q^m}})$, et par $Y^{(b)}$ le lieu des points dans $Y$ dont les multiplicit\'es sont plus grandes ou \'egales \`a $\mu_Y(X_{\f_{q^m}})+1$. De plus, on d\'esigne par $Y^{(a)}(\f_q)$ (\resp $Y^{(b)}(\f_q)$) l'ensemble des points $\f_{q^m}$-rationnels de $Y^{(a)}$ (\resp $Y^{(b)}$) qui apparaissent dans les images inverses des \'el\'ements de $\mathbb P^n_{\f_q}(\f_q)$ par rapport \`a l'immersion ferm\'ee de $Y$ dans $\mathbb P^n_{\f_{q^m}}$ sous le changement de base $\mathbb P^n_{\f_{q^m}}\rightarrow\mathbb P^n_{\f_q}$ (voir la d\'efinition \ref{descent of k-points}). Donc on a $Y(\f_q)=Y^{(a)}(\f_q)\bigsqcup Y^{(b)}(\f_q)$.

D'apr\`es le corollaire \ref{sub}, on obtient que $Y^{(a)}$ est dense dans $Y$ si $Y^{(a)}\neq\emptyset$, et $Y^{(b)}$ est de dimension plus petite ou \'egale \`a $\dim (Y)-1$.

Dans la suite, on construit une famille d'arbres d'intersection $\{\mathscr T_Y\}$, o\`u $Y\in\mathcal C(X_{\f_{q^m}}\cdot X_{1,\f_{q^m}}\intersect X_{n-s-1,{\f_{q^m}}})$. La racine de l'arbre d'intersection $\mathscr T_Y$ est $Y$.

Pour construire les sommets de profondeur plus grande ou \'egale \`a $1$, soit $U$ un sommet d\'eja construit dans les arbres d'intersection $\{\mathscr T_Y\}$. On consid\`ere le sommet $U$ comme un sch\'ema int\`egre. Il faut consid\'erer les propri\'et\'es de $U(\f_q)$, o\`u $U(\f_q)$ est d\'efini dans la d\'efinition \ref{descent of k-points}. Si $U^{(b)}(\f_q)=\emptyset$, le sommet $U$ est une feuille dans les arbres d'intersection.

Si $U^{(b)}(\f_q)\neq\emptyset$, alors on a $\mu_U(X_{\f_{q^m}})<\delta$. D'apr\`es le corollaire \ref{point}, on peut trouver un $h\in\mathcal T^{\delta-\mu_U(X_{\f_{q^m}})}(f)\otimes_{\f_q}\f_{q^m}$, tel que l'hypersurface d\'efinie par le polyn\^ome $h$ intersecte $U$ proprement. Bien s\^ur on a $\deg (h)\leqslant\delta-1$. Dans ce cas-l\`a, on d\'efinit $V(h)$ comme l'\'etiquette du sommet $U$.

Les poids des ar\^etes sont les multiplicit\'es d'intersection respectives.

Pour la construction plus haute, toutes les \'etiquettes mentionn\'ees ci-dessus sont de dimension $n-1$, donc les sommet dans $\mathcal C_w$ sont de dimension $n-w-2$, o\`u $1\leqslant w\leqslant n-2$ est un entier.

Le lemme suivant est une propri\'et\'e de l'ensemble $\mathcal Z_*$ (voir la d\'efinition \ref{Zs}), qui sera utile dans la d\'emonstration du th\'eor\`eme \ref{main result}. C'est la motivation que l'on d\'efinit le sous-ensemble $\mathcal Z_*$ de $\mathcal C_*$.
\begin{lemm}\label{control of singular locus}
  Avec les notations et constructions ci-dessus, pour tout $\xi\in X^{\mathrm{sing}}_{\f_{q^m}}(\f_q)$, il existe au moins un $Z\in\mathcal Z_*$ tel que $\xi\in Z^{(a)}(\f_q)$, o\`u $\mathcal Z_*$ est d\'efini dans la d\'efinition \ref{Zs}.
\end{lemm}
\begin{proof}
Soit $Y\in\mathcal C(X_{\f_{q^m}}\cdot X_{1,\f_{q^m}}\intersect X_{n-s-1,{\f_{q^m}}})$. Par la construction des arbres d'intersection $\mathscr T_Y$ ci-dessus, pour tout $\xi\in X^{\mathrm{sing}}_{\f_{q^m}}(\f_q)$, on a $\xi\in Y(\f_q)$ pour au moins un $Y\in\mathcal C(X_{\f_{q^m}}\cdot X_{1,\f_{q^m}}\intersect X_{n-s-1,{\f_{q^m}}})$.

Soit $\mathcal C_m$ comme dans la d\'efinition \ref{Cs}. Si $\mathcal C_{n-2}\neq\emptyset$, les sommets dans $\mathcal C_{n-2}$ sont certains points rationnels, qui doivent \^etre r\'eguliers. Si $\mathcal C_t=\emptyset$ mais $\mathcal C_{t-1}\neq\emptyset$, alors pour tout $U\in \mathcal C_{t-1}$, on a $U^{(b)}(\f_q)=\emptyset$. Donc pour tout $\xi\in X^{\mathrm{sing}}_{\f_{q^m}}(\f_q)$, il toujours existe un $Y\in\mathcal C_w$, tel que $\xi\in Y^{(a)}(\f_q)$.

Pour un $\xi\in  X_{\f_{q^m}}^{\mathrm{sing}}(\f_q)$ fix\'e, on prend la valeur minimale $w$ telle qu'il existe un $Y\in\mathcal C_w$ v\'erifiant $\xi\in Y^{(a)}(\f_q)$. S'il existe un tel $Y\in\mathcal Z_w$, on a l'assertion. Sinon, pour tout $Y\in\mathcal C_w$ qui satisfait $\xi\in Y^{(a)}(\f_q)$, on a toujours $Y\not\in\mathcal Z_w$. Alors on peut trouver l'entier positif maximal $w'$ qui satisfait la condition suivantes: $w'<w$, et il existe un $Y_0\in \mathcal C_{w'}$ tel que $Y\subsetneq Y_0$ mais $Y$ ne soit pas parmi les descendants de $Y_0$. Si $\xi\in Y_0^{(a)}(\f_q)$, il contredit avec ce que $w$ est minimal. Si $\xi\in Y_0^{(b)}(\f_q)$, alors on a $\mu_Y(X_{\f_{q^m}})=\mu_\xi(X_{\f_{q^m}})\geqslant\mu_{Y^0}(X_{\f_{q^m}})+1$. D'apr\`es la construction des arbres d'intersection ci-dessus, $Y$ est un descendant de $Y_0$, qui contredit avec ce que le choix de $w'$ est maximal.

En r\'esum\'e, on a l'assertion.
\end{proof}
\subsection{D\'emonstration du th\'eor\`eme \ref{main result}}
Avec toutes les pr\'eparations ci-dessus, on va d\'emontrer le th\'eor\`eme \ref{main result}.
\begin{proof}[D\'emonstration du th\'eor\`eme \ref{main result}]
On prend la construction des arbres d'intersection dont les racines sont les \'el\'ements dans $\mathcal C(X_{\f_{q^m}}\cdot X_{1,\f_{q^m}}\intersect X_{n-s-1,\f_{q^m}})$ dans \S \ref{construction of intersection trees}. D'apr\`es la proposition \ref{mult of point under base change}, comme $\f_q$ est un corps parfait, on a $\mu_\xi(X)=\mu_{\xi'}(X_{\f_{q^m}})$, o\`u $\xi\in X(\f_q)$ et $\xi'=\xi\times_{\spec \f_q}\spec \f_{q^m}$.

Donc on obtient
\begin{eqnarray}\label{first step}
  & &\sum\limits_{\xi\in X(\f_q)}\mu_\xi(X)(\mu_\xi(X)-1)^{n-s-1}\\
  &=&\sum\limits_{\xi\in X^{\mathrm{sing}}(\f_q)}\mu_\xi(X)(\mu_\xi(X)-1)^{n-s-1}\nonumber\\
  &=&\sum\limits_{\xi\in X^{\mathrm{sing}}_{\f_{q^m}}(\f_q)}\mu_\xi(X_{\f_{q^m}})(\mu_\xi(X_{\f_{q^m}})-1)^{n-s-1},\nonumber
\end{eqnarray}
o\`u la notation $ X^{\mathrm{sing}}_{\f_{q^m}}(\f_q)$ est introduite dans la d\'efinition \ref{descent of k-points}.

D'apr\`es le lemme \ref{control of singular locus}, pour tout $\xi\in X^{\mathrm{sing}}_{\f_{q^m}}(\f_q)$, on peut trouver un $Z\in\mathcal Z_*$ tel que $\xi\in Z^{(a)}(\f_q)$. Donc on obtient
\begin{eqnarray}\label{final 1}
& &\sum\limits_{\xi\in X^{\mathrm{sing}}_{\f_{q^m}}(\f_q)}\mu_\xi(X_{\f_{q^m}})(\mu_\xi(X_{\f_{q^m}})-1)^{n-s-1}\\
&\leqslant&\sum_{t=0}^{s}\sum_{Z\in\mathcal Z_t}\sum_{\xi\in Z^{(a)}(\f_q)}\mu_\xi(X_{\f_{q^m}})(\mu_\xi(X_{\f_{q^m}})-1)^{n-s-1}.\nonumber
\end{eqnarray}

D'apr\`es le corollaire \ref{mult}, pour tout $Z\in \mathcal Z_*$, on obtient que l'in\'egalit\'e
\[\mu_Z(X_{\f_{q^m}})-1\leqslant\mu_Z(X_{i,\f_{q^m}}),\]
est v\'erifi\'ee pour tout $i=1,\ldots,n-s-1$. Donc on a l'in\'egalit\'e
\begin{equation}\label{final_1'}
\mu_Z(X_{\f_{q^m}})(\mu_Z(X_{\f_{q^m}})-1)^{n-s-1}\leqslant\mu_Z(X_{\f_{q^m}})\mu_Z(X_{1,\f_{q^m}})\cdots\mu_Z(X_{n-s-1,\f_{q^m}}).
\end{equation}

 D'apr\`es la proposition \ref{globale} et l'in\'egalit\'e \eqref{final_1'}, on a
\begin{eqnarray}\label{final 2}
& &\sum_{Z\in\mathcal Z_t}\mu_Z(X_{\f_{q^m}})(\mu_Z(X_{\f_{q^m}})-1)^{n-s-1}\deg(Z)\\
&\leqslant&\sum_{Z\in\mathcal Z_t}\mu_Z(X_{\f_{q^m}})\mu_Z(X_{1,\f_{q^m}})\cdots\mu_Z( X_{n-s-1,\f_{q^m}})\deg(Z)\leqslant\delta(\delta-1)^{n-s+t-1}\nonumber
\end{eqnarray}
pour tout $t=0,\ldots,s$, car toutes les \'etiquettes dans $\mathcal C'_*$ sont de degr\'e plus petit ou \'egal \`a $\delta-1$.

Avec les in\'egalit\'es \eqref{final 1} et \eqref{final 2}, on a
\begin{eqnarray}\label{final step}
  & &\sum_{t=0}^{s}\sum_{Z\in\mathcal Z_t}\sum_{\xi\in Z^{(a)}(\f_q)}\mu_\xi(X_{\f_{q^m}})(\mu_\xi(X_{\f_{q^m}})-1)^{n-s-1}\\
  &=&\sum_{t=0}^{s}\sum_{Z\in\mathcal Z_t}\mu_Z(X_{\f_{q^m}})(\mu_Z(X_{\f_{q^m}})-1)^{n-s-1}\#Z^{(a)}(\f_q)\nonumber\\
  &\leqslant&\sum_{t=0}^{s}\sum_{Z\in\mathcal Z_t}\mu_Z(X_{\f_{q^m}})(\mu_Z(X_{\f_{q^m}})-1)^{n-s-1}\#Z(\f_q)\nonumber\\
  &\leqslant&\sum_{t=0}^{s}\sum_{Z\in\mathcal Z_t}\left(\mu_Z(X_{\f_{q^m}})(\mu_Z(X_{\f_{q^m}})-1)^{n-s-1}\deg(Z)\#\mathbb P^{s-t}(\f_q)\right)\nonumber\\
  &=&\sum_{t=0}^{s}\#\mathbb P^{s-t}(\f_q)\left(\sum_{Z\in\mathcal Z_t}\mu_Z(X_{\f_{q^m}})(\mu_Z(X_{\f_{q^m}})-1)^{n-s-1}\deg(Z)\right)\nonumber\\
   &\leqslant&\delta(\delta-1)^{n-s-1}\#\mathbb{P}^{s}(\f_q)+\delta(\delta-1)^{n-s}\#\mathbb{P}^{s-1}(\f_q)\nonumber\\
  & &+\cdots+\delta(\delta-1)^{n-1},\nonumber
\end{eqnarray}
o\`u l'in\'egalit\'e dans la trois\`eme ligne est v\'erifi\'ee d'apr\`es la proposition \ref{lineaire}, et la derni\`ere int\'egalit\'e est vraie d'apr\`es le lemme \ref{grassmanne}.

D'apr\`es les in\'egalit\'es \eqref{first step}, \eqref{final 1} et \eqref{final step}, on obtient le r\'esultat.
\end{proof}

\begin{rema}
Si $n=2$, par la m\'ethode similaire \`a la d\'emonstration du th\'eor\`eme \ref{main result}, on obtient l'in\'egalit\'e \eqref{intro-fulton-multiplicity} , o\`u en fait on peut consid\'erer tous les points ferm\'es de cette courbe plane. D'apr\`es le th\'eor\`eme \ref{main result}, on a
\begin{eqnarray*}
  \sum\limits_{\xi\in X(\f_q)}\mu_\xi(X)(\mu_\xi(X)-1)^{n-s-1}&\leqslant&(s+1)^2\delta(\delta-1)^{n-s-1}\max\{\delta-1,q\}^{s}\\
  &\ll_n&\delta(\delta-1)^{n-s-1}\max\{\delta-1,q\}^{s},
\end{eqnarray*}
 comme $s\leqslant n-2$.
\end{rema}

\begin{exem}\label{cylinder from a plane curve}
   Soit $X'\hookrightarrow\mathbb{P}^2_{\f_q}$ une courbe plane r\'eduite de degr\'e $\delta$ d\'efinie par l'\'equation homog\`ene $f(T_0,T_1,T_2)=0$ qui a seulement un point $\f_q$-rationnel singulier de multiplicit\'e $\delta$. Alors on peut consid\'erer $f(T_0,T_1,T_2)$ comme un polyn\^ome homog\`ene de degr\'e $\delta$ dans $\f_q[T_0,\ldots,T_n]$. Donc l'\'equation homog\`ene $f(T_0,T_1,T_2)=0$ d\'efinit une hypersurface r\'eduite de degr\'e $\delta$ de $\mathbb{P}^n_{\f_q}$ ($n\geqslant2$), not\'ee comme $X$ cette hypersurface. Soit $[a_0:a_1:a_2]$ la coordonn\'ee projective du point singulier de $X'$. Alors on a
   \begin{eqnarray*}X^{\mathrm{sing}}(\f_q)&=&\{[x_0:\cdots:x_n]\in\mathbb P^n_{\f_q}(\f_q)|\;x_0=a_0,x_1=a_1,x_2=a_2\}\cup\\
   	& &\{[x_0:\cdots:x_n]\in\mathbb P^n_{\f_q}(\f_q)|\;x_0=x_1=x_2=0\},\end{eqnarray*}
    o\`u tous les points $\f_q$-rationnels singuliers sont de multiplicit\'e $\delta$. Alors pour l'hypersurface $X$, on obtient
\begin{eqnarray*}
  \sum\limits_{\xi\in X(\f_q)}\mu_\xi(X)(\mu_\xi(X)-1)&=&\delta(\delta-1)q^{n-2}+\delta(\delta-1)(q^{n-3}+\cdots+1)\\
  &=&\delta(\delta-1)(q^{n-2}+\cdots+1)\\
  &\;\sim_n&\delta(\delta-1)q^{n-2}.
\end{eqnarray*}
Alors l'ordre de $\delta$ et l'ordre de $q$ dans le th\'eor\`eme \ref{main result} sont optimaux pour le cas o\`u $q$ est assez grand et $\dim (X^{\mathrm{sing}})=n-2$.
\end{exem}
\begin{rema}\label{proper counting function}
  Soit $X$ une hypersurface de $\mathbb P^n_{\f_q}$, o\`u $\dim(X^{\mathrm{sing}})=s$. D'apr\`es le th\'eor\`eme \ref{main result}, on obtient
\begin{eqnarray*}
   & &\sum\limits_{\xi\in X(\f_q)}\mu_\xi(X)(\mu_\xi(X)-1)\\
   &\leqslant&\sum\limits_{\xi\in X(\f_q)}\mu_\xi(X)(\mu_\xi(X)-1)^2\leqslant\cdots\leqslant\sum\limits_{\xi\in X(\f_q)}\mu_\xi(X)(\mu_\xi(X)-1)^{n-s-1}\\
   &\leqslant&\delta(\delta-1)^{n-s-1}(q^{s}+q^{s-1}+\cdots+1)+\\
    & &\delta(\delta-1)^{n-s}(q^{s-1}+q^{s-2}+\cdots+1)+\cdots+\delta(\delta-1)^{n-1}.
\end{eqnarray*}
Donc on obtient que pour tout $t\in\{1,\ldots,n-s-1\}$, on a
\[\sum\limits_{\xi\in X(\f_q)}\mu_\xi(X)(\mu_\xi(X)-1)^t\ll_n\delta(\delta-1)^{n-s-1}q^s\]
lorsque $q\geqslant\delta-1$.

Soit $t$ un entier avec $t\geqslant n-s$, on peut construire un exemple (l'exemple \ref{cylinder from a plane curve} par exemple), tel que
\[\sum\limits_{\xi\in X(\f_q)}\mu_\xi(X)(\mu_\xi(X)-1)^t\sim_n\delta(\delta-1)^tq^s\]
lorsque $q\geqslant\delta-1$.

Soit $f(T)\in\mathbb R[T]$ un polyn\^ome de degr\'e $n-s$, qui satisfait $f(1)=0$ et $f(x)>0$ pour tout $x\geqslant2$. Donc il existe une constante $C_f>0$ d\'ependant du polyn\^ome $f(T)$, telle que
\[f(x)\leqslant C_fx(x-1)^{n-s-1}\]
pour tout $x\geqslant1$. D'o\`u l'on a
\[\sum_{\xi\in X(\f_q)}f(\mu_\xi(X))\leqslant C_f\sum_{\xi\in X(\f_q)}\mu_\xi(X)(\mu_\xi(X)-1)^{n-s-1}\ll_{n,f}\delta(\delta-1)^{n-s-1}\max\{\delta-1,q\}^s.\]
 Alors le choix de la fonction de comptage
\[\mu_\xi(X)(\mu_\xi(X)-1)^{n-s-1}\]
 est convenable pour d\'ecrire la complexit\'e du lieu singulier de $X$, o\`u $\xi\in X(\f_q)$.
\end{rema}
Pour g\'en\'eraliser le th\'eor\`eme \ref{main result} au cas o\`u $X$ est un sch\'ema projectif g\'en\'eral, on a la conjecture suivante.
\begin{conj}
  Soit $X$ un sous-sch\'ema ferm\'e r\'eduit de $\mathbb P^n_{\f_q}$ qui est de dimension pure $d$ et de degr\'e $\delta$. Si la dimension du lieu singulier de $X$ est $s$, alors on a
  \begin{equation*}
  \sum\limits_{\xi\in X(\f_q)}\mu_\xi(X)(\mu_\xi(X)-1)^{d-s}\ll_n\delta(\delta-1)^{d-s}q^s.
\end{equation*}
\end{conj}

\backmatter

\bibliography{liu}

\def\cprime{$'$} \def\cprime{$'$}
\providecommand{\bysame}{\leavevmode ---\ }
\providecommand{\og}{``}
\providecommand{\fg}{''}
\providecommand{\smfandname}{\&}
\providecommand{\smfedsname}{\'eds.}
\providecommand{\smfedname}{\'ed.}
\providecommand{\smfmastersthesisname}{M\'emoire}
\providecommand{\smfphdthesisname}{Th\`ese}
\begin{thebibliography}{10}

\bibitem{AM_commutative_algebra}
{\scshape M.~F. Atiyah {\normalfont \smfandname} I.~G. Macdonald} --
  \emph{Introduction to commutative algebra}, Addison-Wesley Publishing Co.,
  Reading, Mass.-London-Don Mills, Ont., 1969.

\bibitem{Bourbaki83}
{\scshape N.~Bourbaki} -- \emph{\'{E}l\'ements de math\'ematique}, Masson,
  Paris, 1983, Alg\`ebre commutative. Chapitre 8. Dimension. Chapitre 9.
  Anneaux locaux noeth\'eriens complets. [Commutative algebra. Chapter 8.
  Dimension. Chapter 9. Complete Noetherian local rings].

\bibitem{GTM150}
{\scshape D.~Eisenbud} -- \emph{Commutative algebra}, Graduate Texts in
  Mathematics, vol. 150, Springer-Verlag, New York, 1995, With a view toward
  algebraic geometry.

\bibitem{Fulton2}
{\scshape W.~Fulton} -- \emph{Algebraic curves. {A}n introduction to algebraic
  geometry}, W. A. Benjamin, Inc., New York-Amsterdam, 1969, Notes written with
  the collaboration of Richard Weiss, Mathematics Lecture Notes Series.

\bibitem{Fulton}
\bysame , \emph{Intersection theory}, second \smfedname, Ergebnisse der
  Mathematik und ihrer Grenzgebiete. 3. Folge. A Series of Modern Surveys in
  Mathematics [Results in Mathematics and Related Areas. 3rd Series. A Series
  of Modern Surveys in Mathematics], vol.~2, Springer-Verlag, Berlin, 1998.

\bibitem{Gabber_Liu_Lorenzini1}
{\scshape O.~Gabber, Q.~Liu {\normalfont \smfandname} D.~Lorenzini} -- {\og The
  index of an algebraic variety\fg}, \emph{Inventiones Mathematicae}
  \textbf{192} (2013), no.~3, p.~567--626.

\bibitem{Gabber_Liu_Lorenzini2}
\bysame , {\og Hypersurfaces in projective schemes and a moving lemma\fg},
  \emph{Duke Mathematical Journal} \textbf{164} (2015), no.~7, p.~1187--1270.

\bibitem{EGAIV_2}
{\scshape A.~Grothendieck {\normalfont \smfandname} J.~Dieudonn\'e} -- {\og
  \'{E}l\'ements de g\'eom\'etrie alg\'ebrique. {IV}. \'{E}tude locale des
  sch\'emas et des morphismes de sch\'emas. {II}\fg}, \emph{Institut des Hautes
  \'Etudes Scientifiques. Publications Math\'ematiques} (1965), no.~24, p.~231.

\bibitem{NouveauEGA1}
\bysame , \emph{\'{E}l\'ements de g\'eom\'etrie alg\'ebrique}, Die Grundlehren
  der mathematischen Wissenschaften in Einzeldarstellungen, vol. 166,
  Springer-Verlag, Berlin, 1971.

\bibitem{GTM52}
{\scshape R.~Hartshorne} -- \emph{Algebraic geometry}, Springer-Verlag, New
  York, 1977, Graduate Texts in Mathematics, No. 52.

\bibitem{Kollar2007}
{\scshape J.~Koll{\'a}r} -- \emph{Lectures on resolution of singularities},
  Annals of Mathematics Studies, vol. 166, Princeton University Press,
  Princeton, NJ, 2007.

\bibitem{Laumon1975}
{\scshape G.~Laumon} -- {\og Degr\'e de la vari\'et\'e duale d'une hypersurface
  \`a singularit\'es isol\'ees\fg}, \emph{Bulletin de la Soci\'et\'e
  Math\'ematique de France} \textbf{104} (1976), no.~1, p.~51--63.

\bibitem{LiuQing}
{\scshape Q.~Liu} -- \emph{Algebraic geometry and arithmetic curves}, Oxford
  Graduate Texts in Mathematics, vol.~6, Oxford University Press, Oxford, 2002,
  Translated from the French by Reinie Ern{\'e}, Oxford Science Publications.

\bibitem{Matsumura1}
{\scshape H.~Matsumura} -- \emph{Commutative algebra}, second \smfedname,
  Mathematics Lecture Note Series, vol.~56, Benjamin/Cummings Publishing Co.,
  Inc., Reading, Mass., 1980.

\bibitem{Mazur_1975_ag}
{\scshape B.~Mazur} -- {\og Eigenvalues of {F}robenius acting on algebraic
  varieties over finite fields\fg}, in \emph{Algebraic geometry ({P}roc.
  {S}ympos. {P}ure {M}ath., {V}ol. 29, {H}umboldt {S}tate {U}niv., {A}rcata,
  {C}alif., 1974)}, Amer. Math. Soc., Providence, R.I., 1975, p.~231--261.

\bibitem{Milne}
{\scshape J.~S. Milne} -- \emph{\'{E}tale cohomology}, Princeton Mathematical
  Series, vol.~33, Princeton University Press, Princeton, N.J., 1980.

\bibitem{Mustata-notes}
{\scshape M.~Musta\c{t}\u{a}} -- {\og Zeta functions in algebraic geometry\fg},
  \url{http://www.math.lsa.umich.edu/~mmustata/zeta_book.pdf}, 2011.

\bibitem{Nagata62}
{\scshape M.~Nagata} -- \emph{Local rings}, Interscience Tracts in Pure and
  Applied Mathematics, No. 13, Interscience Publishers a division of John Wiley
  \& Sons\, New York-London, 1962.

\bibitem{Nowak_1998}
{\scshape K.~J. Nowak} -- {\og On the intersection multiplicity of images under
  an etale morphism\fg}, \emph{Colloquium Mathematicum} \textbf{75} (1998),
  no.~2, p.~167--174.

\bibitem{Poonen_2009_CT}
{\scshape B.~Poonen} -- {\og Existence of rational points on smooth projective
  varieties\fg}, \emph{Journal of the European Mathematical Society (JEMS)}
  \textbf{11} (2009), no.~3, p.~529--543.

\bibitem{Roberts98}
{\scshape P.~C. Roberts} -- \emph{Multiplicities and {C}hern classes in local
  algebra}, Cambridge Tracts in Mathematics, vol. 133, Cambridge University
  Press, Cambridge, 1998.

\bibitem{SamuelLocAlg}
{\scshape P.~Samuel} -- \emph{Alg\`ebre locale}, M\'emorial des Sciences
  Math\'emathiques, fascicule 123, Gauthier-Villars, Paris, 1953.

\bibitem{Samuel}
\bysame , \emph{M\'ethodes d'alg\`ebre abstraite en g\'eom\'etrie
  alg\'ebrique}, Seconde \'edition, corrig\'ee. Ergebnisse der Mathematik und
  ihrer Grenzgebiete, Band 4, Springer-Verlag, Berlin-New York, 1967.

\bibitem{SerreLocAlg}
{\scshape J.-P. Serre} -- \emph{Local algebra}, Springer Monographs in
  Mathematics, Springer-Verlag, Berlin, 2000, Translated from the French by
  CheeWhye Chin and revised by the author.

\bibitem{Stanley}
{\scshape R.~P. Stanley} -- \emph{Enumerative combinatorics. {V}olume 1},
  second \smfedname, Cambridge Studies in Advanced Mathematics, vol.~49,
  Cambridge University Press, Cambridge, 2012.

\bibitem{FGA-Weil}
{\scshape A.~Weil} -- \emph{Foundations of {A}lgebraic {G}eometry}, American
  Mathematical Society Colloquium Publications, vol. 29, American Mathematical
  Society, New York, 1946.

\end{thebibliography}
\bibliographystyle{smfplain}

\end{document}